\documentclass[3p,times]{elsarticle}

\usepackage{amsmath}
\usepackage{amsfonts}
\usepackage{amsthm}

\newtheorem{theorem}{Theorem}
\theoremstyle{definition} 
\newtheorem{remark}{Remark}

\newcommand{\N}{\mathbb{N}}
\newcommand{\R}{\mathbb{R}}
\newcommand{\CFL}{\textrm{CFL}}
\newcommand{\icomplex}{\mathfrak{i}}
\newcommand{\es}[1]{e{#1}}
\newcommand{\dx}{\mathrm{d}x}
\newcommand{\dt}{\mathrm{d}t}
\newcommand{\ds}{\mathrm{d}s}

\usepackage{tikz}
\usepackage{subcaption}
\usepackage{adjustbox} 

\definecolor{plotgreen}{rgb}{0.0, 1.0, 0.5}

\usepackage[rightcaption]{sidecap}
\sidecaptionvpos{figure}{m}

\ifpdf
\DeclareGraphicsExtensions{.eps,.pdf,.png,.jpg}
\else
\DeclareGraphicsExtensions{.eps}
\fi

\usepackage{booktabs}

\begin{document}
	
\begin{frontmatter}
	
\journal{Applied Mathematics and Computation} 

\title{Stability analysis of Arbitrary-Lagrangian-Eulerian ADER-DG methods\\on classical and degenerate spacetime geometries}

\cortext[mycorrespondingauthor]{Corresponding author}

\author[Verona]{Mauro Bonafini\corref{mycorrespondingauthor}}
\ead{mauro.bonafini@univr.it}
\author[Roma]{Davide Torlo}
\ead{davide.torlo@uniroma1.it}
\author[Verona]{Elena Gaburro}
\ead{elena.gaburro@univr.it}

\address[Verona]{Department of Computer Science, University of Verona, Strada le Grazie 15, Verona, 37134, Italy}
\address[Roma]{Mathematics Department ``Guido Castelnuovo'', University of Rome Sapienza, p.le Aldo Moro, 5, Rome, 00189, Italy}

%
\begin{abstract}
	In this paper, we present a thorough von Neumann stability analysis of explicit and implicit Arbitrary-Lagrangian-Eulerian (ALE) ADER discontinuous Galerkin (DG) methods on \emph{classical} and \emph{degenerate} spacetime \emph{geometries} for hyperbolic equations.

	First, we rigorously study $\CFL$ stability conditions for the \emph{explicit} ADER-DG method, confirming results widely used in the literature while specifying their limitations.
	Moreover, we discuss stability bounds for ALE methods and characterise the admissible range of grid velocities once a target $\CFL$ is fixed.

	Next, we extend the stability study to ADER-DG in the presence of \emph{degenerate spacetime elements}, 
	with \emph{zero size} at the beginning and the end of the time step, but with a \emph{non zero spacetime volume}. 
	This kind of elements has been introduced in a series of articles on direct ALE methods by Gaburro \emph{et al.} to connect via spacetime control volumes regenerated Voronoi tessellations after a topology change. 
	Here, we imitate this behaviour in a 1d surrogate setting by fictitiously inserting degenerate elements in between two cells.
	We show that over this simplified degenerate spacetime geometry, both for the explicit and implicit ADER-DG, the von Neumann analysis leads to the \emph{same} $\CFL$ \emph{stability conditions} as those for classical geometries, laying the theoretical foundations for their use in the context of ALE methods.
\end{abstract}
%
\end{frontmatter}


\section{Introduction}

The spacetime predictor-corrector ADER approach, 
introduced in its original formulation in~\cite{mill,toro1}, extended to nonlinear systems in~\cite{toro4,toro3} and to unstructured domains in~\cite{Kaser2003,Kaser05,DumbserKaeser07,CastroToro}, 
and finally presented with a modern and effective approach in~\cite{dumbser2008unified},
is nowadays widely used to reach high order of accuracy also in time when solving hyperbolic partial differential equations (PDEs). 
It consists of an element-local iterative method to construct 
a piece-wise polynomial, called predictor, that represents an approximation of high order of accuracy both in space and time of the PDE solution. Then, in the corrector phase, the predictor is directly inserted into a one-step update formula to evolve the solution from one time step to the next one.
The method has been shown to be very reliable and effective since it introduces minimal dissipation and can be parallelized with minimal communication overhead~\cite{ADERGRMHD,dumbser2018efficient}.

In particular, ADER approaches can be used within the framework of discontinuous Galerkin (DG) finite element methods, as originally proposed in~\cite{QiuDumbserShu,dumbser_jsc,lorcher2007discontinuous, Gassner2011a}.
Later, the ADER-DG methods have been applied to many different hyperbolic systems,
starting from the classical Euler and shallow water equations \cite{rannabauer2018ader,fernandez2022arbitrary,rio2024high,ciallella2024very}, 
multiphase models \cite{rio2024high,kemm2020simple}, magnetohydrodynamics  \cite{balsara2012self,fambri2020discontinuous},  
dispersive, turbulent and reactive flows \cite{busto2021high,yuan2022hybrid,busto2022new,popov2023space},  
up to the study of much more complex systems, as unified models for continuum mechanics \cite{dumbser2016high,dumbser2017high,tavelli2020space,busto2020high,chiocchetti2021high,lakiss2024ader}
and various first order reformulations of the Einstein field equations for general relativity \cite{dumbser2024well,zanotti2025new,muzzolon2025high}.

Over the years, many authors have worked with ADER-DG schemes 
proposing performance improvements~\cite{veiga2024improving,micalizzi2025efficient,wolf2020optimization,popov2024effective,marot2025mixed}, 
implicit~\cite{dumbser2016space} and semi-implicit~\cite{offner2025analysis,boscheri2022construction,popov2025theoryinternalstructureaderdg} formulations,
and in particular their use in the context of \emph{direct} Arbitrary-Lagrangian-Eulerian (ALE) 
methods~\cite{boscheri2017arbitrary}.
In this latter framework, the mesh moves following as close as possible the local fluid flow and, 
in order to evolve the PDE from one mesh to the next one, 
an effective approach consists in connecting them via spacetime control volumes, 
possibly degenerate~\cite{gaburro2020high,gaburro2021bookchapter,gaburro2025high,Gaburro2027}, over which integrating the PDE.
Here, ADER-DG methods offer the ideal tools to build a one-step solver 
that is automatically able to integrate the PDE in space and time and to perform \emph{directly} the time update
(without any need for projection and reconstruction procedures, typical of \emph{indirect} approaches \cite{ReALE2010, ReALE2015}).

Given the widespread adoption of ADER and ADER-DG methods, a rigorous theoretical investigation of their fundamental properties becomes necessary. While the convergence of the local spacetime predictor has been established for linear problems in \cite{jackson2017eigenvalues} and extended to nonlinear systems in \cite{busto2020high}, stability studies still require further investigation.
Regarding stability, some authors have addressed ADER-type schemes in the simplified framework of ordinary differential equations \cite{han2021dec, popov2025theoryinternalstructureaderdg, popov2025high}, but a comprehensive analysis in the PDE setting is still missing. The seminal work~\cite{dumbser2008unified} provides stability considerations only for low polynomial degrees, while for higher orders the admissible $\CFL$ numbers are typically determined empirically. Moreover, to the best of our knowledge, no dedicated analysis has been carried out for the newly introduced ALE ADER-DG methods on possibly degenerate spacetime geometries.

For the above reasons, it is of general interest to establish a rigorous stability analysis of ADER-DG methods.
There are several strategies to study the stability of numerical methods, some of which apply only to linear problems, while others are also able to treat nonlinear problems. 
For linear problems, one can look at the Lax--Richtmyer stability \cite{leveque2007finite}, which, however, turns out to be difficult to prove in many situations. In the context of periodic domains,
the von Neumann stability analysis guarantees Lax--Richtmyer stability and can be easily performed using Fourier modes instead of fully discretized problems, strongly reducing the complexity of the stability analysis.
Indeed, there are many works that have studied the von Neumann stability to deduce $\CFL$ conditions or other parameter conditions, e.g. for residual distribution in 1d and 2d \cite{michel2021spectral,michel2023spectral} and for DG in 1d \cite{sherwin2000dispersion}.

The extension of linear analyses to nonlinear problems is not straightforward, but the results obtained in the linear setting are often a good indicator of the $\CFL$ that should be used in the presence of non linearities.
A direct study of the $L^2$ stability can be done also in the nonlinear case, but it would still require an analysis of all the degrees of freedom and it is typically easy to apply only to simple time discretizations (explicit/implicit Euler, SSPRK) \cite{cockburn1998local,liu20082}.
Another way to tackle the stability would be to mimic the integration by parts at the discrete level to guarantee energy or entropy stability in the spatial discretization \cite{friedrich2019entropy_sbp,offner2018stability}, and also to apply relaxation in time to preserve or dissipate a global entropy \cite{ranocha2020relaxation,abgrall2022relaxation,gaburro2023high}.

\subsection{Aim and structure of this paper}
The scope of this paper is to fill the theoretical gap in ADER-DG methods by providing a thorough von Neumann stability analysis,
accounting for both the explicit and implicit formulations on fixed or moving meshes.
In Section~\ref{sec.method-description}, to fix the notation and to make the paper self-contained, 
we present the ADER-DG scheme written in the so-called direct Arbitrary-Lagrangian-Eulerian framework. We apply it to scalar nonlinear hyperbolic PDEs on a classical 1d+time geometry. 
Next, in Section~\ref{sec.stability-classical}, we perform the von Neumann stability analysis of the method and show numerical results for its convergence  order.
Then, in Section~\ref{sec.degenerate}, we introduce into our computational domain \emph{fictitious spacetime degenerate elements}
that mimic the behaviour of the hole-like sliver elements used in~\cite{gaburro2020high,gaburro2021unified,gaburro2021bookchapter,gaburro2025high} to connect meshes with topology changes, and we study the stability and convergence  of the method in their presence.
Finally, in Section~\ref{sec.conclusions}, we draw our conclusions and we provide an outlook on future developments.

\section{Explicit and implicit formulation of ALE ADER-DG methods}  
\label{sec.method-description}

Let us consider a first order nonlinear scalar hyperbolic PDE
\begin{equation}
	\label{eq.PDE}
	\partial_t Q(x,t) + \partial_x f(Q(x,t)) = 0, 
\end{equation}
where $Q$ denotes the conserved scalar variable and $f$ denotes the flux function. 

\newpage
\noindent
\emph{Geometry description.}

\noindent
{\linespread{1.1}\selectfont	
Given an initial domain $\Omega^0 = [x_L, x_R]$, with $x_L, x_R \in \R$, we consider $N_e+1$ points $x_L = x_{1/2}^0 < x_{3/2}^0 < \dots < x_{N_e+1/2}^0 = x_R$ and define $N_e$ spatial elements $\Omega_i^0 := [x_{i-1/2}^0, x_{i+1/2}^0]$, $i = 1, \dots, N_e$.
Given an initial time $t^0 \in \R$, we discretize time by means of intervals $[t^n,t^{n+1}]$ with $n\in \mathbb N$ of step size $\Delta t^n = t^{n+1}-t^{n} > 0$ for $n \geq 0$. When not ambiguous, we will simply use the notation $\Delta t$ for $\Delta t^n$. 
In a Lagrangian perspective, we also assume the spatial computational domain to depend on $t$: at each time $t^n$ our computational domain is then denoted as $\Omega^n = [x_L^n, x_R^n]$, and we consider again $N_e+1$ points $x_L^n = x_{1/2}^n < x_{3/2}^n < \dots < x_{N_e+1/2}^n = x_R^n$ and define $N_e$ elements $\Omega_i^n := [x_{i-1/2}^n, x_{i+1/2}^n]$, $i = 1, \dots, N_e$. We denote the length of each element by $\Delta x_i^n := |\Omega_i^n| = (x_{i+1/2}^n - x_{i-1/2}^n)$ and define its barycentre $b_i^n := (x_{i+1/2}^n + x_{i-1/2}^n)/2$.
For each $i = 1,\dots,N_e$, we also define $\Sigma_{i\pm 1/2}^n$ as the segment connecting $(x_{i\pm1/2}^n, t^n)$ to $(x_{i\pm1/2}^{n+1}, t^{n+1})$ and introduce the spacetime control volume $C_i^n$  as the polygonal set with vertices $(x_{i-1/2}^n, t^n), (x_{i+1/2}^n, t^n), (x_{i+1/2}^{n+1}, t^{n+1}), (x_{i-1/2}^{n+1}, t^{n+1})$. In particular, $\partial C_i^n = \Sigma_{i-1/2}^n \cup (\Omega_i^n  \times \{t^n\})  \cup \Sigma_{i+1/2}^n \cup (\Omega_i^{n+1}  \times \{t^{n+1}\})$. We refer to Figure \ref{fig.geometry} for a visual depiction of the setting.
\par}

\begin{figure}[t]
	\centering
	\begin{tikzpicture}
		
		
		\draw[black] (-3.5,0) -- (6,0) -- (6.6,4) -- (-3.2,4) -- (-3.5,0);
		\draw[dotted] (-4,0)   -- (-3.5,0);
		\draw[dotted] (6,0)    -- (6.5,0);
		\draw[dotted] (6.6,4)  -- (7.1,4);
		\draw[dotted] (-3.7,4) -- (-3.2,4);
		\draw[black] (3,0) -- (3.5,4);
		\draw[black] (0,0) -- (-0.2,4);
		
		\node at  (-1.75,-0.3) {$\Omega_{i-1}^n$};
		\node at  (1.5,-0.3) {$\Omega_i^n$};
		\node at  (4.5,-0.3) {$\Omega_{i+1}^n$};
		
		\node at (-0.2,4.3) {$x_{i-\frac12}^{n+1}$};
		\node at  (-1.75,4.3) {$\Omega_{i-1}^{n+1}$};
		\node at  (2,4.3) {$\Omega_i^{n+1}$};
		\node at  (5,4.3) {$\Omega_{i+1}^{n+1}$};
		
		\node at  (0.3,2.8) {$\Sigma_{i-\frac12}^{n}$};
		\node at  (3.6,1.2) {$\Sigma_{i+\frac12}^{n}$};
		
		\node at (0,-0.3) {$x_{i-\frac12}^n$};
		\node at (3,-0.3) {$x_{i+\frac12}^n$};
		\node at (3.5,4.3) {$x_{i+\frac12}^{n+1}$};
		
		\node at (-1.8,2) {$C_{i-1}^n$};
		\node at (1.5,2) {$C_{i}^n$};
		\node at (5,2) {$C_{i+1}^n$};
		
		\node at (6.8,0) {$t^n$};
		\node at (7.6,4) {$t^{n+1}$};
	\end{tikzpicture}
	\caption{Prototype configuration of a classical spacetime domain discretization. At each time $t^n$ the spatial domain is subdivided into elements $\Omega_i^n$, connected in spacetime by control volumes $C_i^n$. Each control volume $C_i^n$ neighbours the previous and next control volume through the lateral interfaces $\Sigma_{i-1/2}^n$ and $\Sigma_{i+1/2}^n$ respectively.}
	\label{fig.geometry}
\end{figure}

\medskip
\noindent
\emph{Families of basis functions.}

\noindent
We fix a polynomial degree $N \geq 0$. For each $n \geq 0$ and $i \in \{1,\dots,N_e\}$, we introduce three different sets of basis functions. First, we define over each domain $\Omega_i^n$ the set of \emph{spatial} basis functions $\{\phi_{i,\ell}^n\}_\ell$. These are modal basis functions centred at the cell barycentre and scaled by the element length, defined as
\begin{equation}\label{eq.spacebasis}
	\phi_{i,\ell}^n \colon \Omega_i^n \to \R, \,x \mapsto \left(\frac{x - b_i^n}{\Delta x_{i}^n} \right)^\ell \quad \text{for } \ell = 0, \dots, N.
\end{equation}
Next, we define over each spacetime control volume two different sets of basis functions. First, we define the set of \emph{moving} basis functions $\{\psi_{i,\ell}^n\}_\ell$, obtained by tracking the spatial basis functions along the trajectory of the cell barycentre from $t^n$ to $t^{n+1}$. They are defined as
\begin{equation}\label{eq.movingbasis}
	\begin{aligned}
		& \psi_{i,\ell}^n \colon C_i^n \to \R, \, (x,t) \mapsto \left( \frac{x - \tilde b_i^n(t)}{\Delta x_{i}^n} \right)^{\ell} \quad \text{for } \ell = 0,\dots,N,
		\\[0.6ex]
		& \text{with } \ \tilde b_i^n(t) = \left(1-\frac{t-t^n}{\Delta t^n}\right) b_i^{n} + \frac{t-t^n}{\Delta t^n} b_i^{n+1}.
	\end{aligned}
\end{equation}
Next, we introduce the set of \emph{spacetime} basis functions $\{\theta_{i,\ell}^n\}_\ell$. These are modal basis functions centred at $(b_i^n, t^n)$ and scaled by the element spacetime lengths, defined as
\begin{equation}\label{eq.spacetimebasis}
	\begin{aligned}
		&\theta_{i,\ell}^n \colon C_i^n \to \R, \, (x,t) \mapsto \left(\frac{x - b_i^n}{\Delta x_{i}^n} \right)^{\ell_1} \, \left( \frac{t-t^n}{\Delta t^n}\right)^{\ell_2}
		\\[0.6ex]
		&\text{for } \ell = \ell_1 + \ell_2(2N + 3 - \ell_2)/2, \,\, 0 \leq \ell_1 + \ell_2 \leq N,
		\\[0.6ex]
		&\text{so that } \ell = 0,\dots, N_{st}-1, \text{ with } N_{st} :=(N+1)(N+2)/2.
	\end{aligned}
\end{equation}
The chosen indexing first considers purely spatial functions and progressively incorporates higher powers in time, getting this way the full set of $N_{st}$ spacetime basis functions up to total degree $N$. For $N = 4$, up to shifting and scaling, the indexing scheme and the corresponding basis functions are given as follows:
\begin{center}
	\begin{tabular}{r|lllll}
		& \multicolumn{5}{c}{{space index} $\ell_1$} \\
		{time index} $\ell_2$ & $\phantom{1}${0} & $\phantom{1}${1} & $\phantom{1}${2} & $\phantom{1}${3} & $\phantom{1}${4} \\ 
		\hline 
		\rule{0pt}{3ex} 
		{0} & $\phantom{1}0: 1$    & $\phantom{1}1: x$    & $\phantom{1}2: x^2$    & $\phantom{1}3: x^3$    & $\phantom{1}4: x^4$ \\
		{1} & $\phantom{1}5: t$    & $\phantom{1}6: xt$   & $\phantom{1}7: x^2t$   & $\phantom{1}8: x^3t$  &           \\
		{2} & $\phantom{1}9: t^2$  & $10: xt^2$ & $11: x^2t^2$ &            &           \\
		{3} & $12: t^3$  & $13: xt^3$ &            &            &           \\
		{4} & $14: t^4$ &           &            &            &           \\
	\end{tabular}
\end{center}

\medskip
\noindent
\emph{Representation of the approximate solution.}

\noindent
The conserved variable $Q$ is represented over each grid element $\Omega_i^n$ via a piecewise polynomial function $u_i^n$ in the form
\begin{equation}\label{eq.un}
	{u}_i^n(x)  = \sum_{\ell = 0}^{N} \phi_{i,\ell}^n(x) \, \hat{{u}}^{n}_{i,\ell} \quad  \text{for } x \in \Omega_i^n, \,i \in \{1,\dots, N_e\}, n \geq 0,
\end{equation}
where $\{ \hat{u}^{n}_{i,\ell} \}_{n,i,\ell} \subset \R$ are the degrees of freedom. We also write $\hat{u}_i^n \in \mathbb R^{N+1}$ to identify the corresponding vector of coefficients. The family $\{u_i^n\}_i$ defines a global (discontinuous) approximant $u^n$ over $\Omega^n$, which is well-defined everywhere except at the interfaces between elements. Approximations defined on each control volume $C_i^n$ will usually take the form of a spacetime polynomial function $q_i^n$ written as
\begin{equation}\label{eq.qn}
	q_i^n(x,t) = \sum_{\ell = 0}^{N_{st}-1} \theta_{i,\ell}^n (x, t) \hat{q}_{i,\ell}^n \quad \text{for } (x,t) \in C_i^n, \,i \in \{1,\dots, N_e\}, n \geq 0,
\end{equation}
where $\{ \hat{q}^{n}_{i,\ell} \}_{n,i,\ell} \subset \R$ are the degrees of freedom. Again, we write $\hat{q}_i^n\in \mathbb R^{N_{st}}$ to identify the corresponding vector of coefficients.

\subsection{Explicit ALE ADER-DG method}
Given an approximate solution $u^n$ in the form \eqref{eq.un}, the explicit ALE ADER-DG method is a two steps method aiming to compute the next approximant $u^{n+1}$. These steps are termed \emph{predictor} and \emph{corrector} step.

\subsubsection{Predictor step}
\label{sec.predictor}
The aim of this step is to build a \emph{local} approximation of the governing PDE~\eqref{eq.PDE} over each spacetime control volume $C_i^n$, using $u_i^n$ as initial condition at time $t^n$. Over a given spacetime control volume $C_i^n$, $i = 1,\dots, N_e$, we seek a spacetime polynomial function $q_i^n$ of the form \eqref{eq.qn} whose coefficients $\hat{q}_i^n$ are to be determined so that $q_i^n$ approximates the element-local PDE problem
\begin{equation*}
	\partial_t q_i^n(x,t) + \partial_x f(q_i^n(x,t)) = 0 \quad \text{over } C_i^n,
\end{equation*}
with $u_i^n$ to be considered as an inflow condition on $\Omega_i^n$. We fix now any spacetime test function $\theta_{i,k}^n$, $k = 0,\dots,N_{st}-1$, multiply the above PDE by $\theta_{i,k}^n$ and integrate over the given control volume $C_i^n$ to get
\begin{equation}\label{eqn.pred-step1}
	\int_{C_i^n} \theta_{i,k}^n(x,t) \left[ \partial_t q_i^n(x,t) + \partial_x f(q_i^n(x,t)) \right]\,\dx \dt = 0 \quad \text{for } k = 0,\dots, N_{st}-1.
\end{equation}
Then, we rewrite the first term in \eqref{eqn.pred-step1} by taking into account a potential jump of $q_i^n$ at the boundary $\Omega_i^n$ of $C_i^n$
via a simplified path-conservative approach~\cite{Pares2006,Castro2006,Castro2008},
obtaining
\begin{equation}\label{eqn.pred-step2}
	\begin{aligned} 
		\int_{C_i^n} \theta_{i,k}^n(x,t) \partial_t q_i^n(x,t) \,\dx \dt + 
		\int_{\Omega_i^n} \theta_{i,k}^n(x,t^n) \, \left( q_i^{n}(x,t^n) - u_i^n(x) \right) \,\dx  & \\[1pt]
		+ \int_{C_i^n} \theta_{i,k}^n(x,t) \partial_x f(q_i^n(x,t)) \,\dx \dt = \, 0
		\qquad \text{for } k = 0,\dots, N_{st}-1.
	\end{aligned} 
\end{equation}
For a given $q_i^n$, we now define the vector $\hat{f}_{i}^n = (\hat{f}_{i,0}^n, \dots, \hat{f}_{i,N_{st}-1}^n)$ so that
\begin{equation}\label{eq.fhat}
	f_i^n(x,t) = \sum_{\ell = 0}^{N_{st}-1} \theta_{i,\ell}^n (x, t) \hat{f}_{i,\ell}^n \qquad 
	\begin{aligned}
		&\text{is the $L^2$-projection of $f(q_i^n)$ over the finite}\\
		&\text{dimensional space $\textup{span}\{\theta_{i,0}^n,\dots,\theta_{i,N_{st}-1}^n\}$}.
	\end{aligned}
\end{equation}
Finally, by replacing in \eqref{eqn.pred-step2} the term $f(q_i^n)$ with $f_i^n$, we get
\begin{equation}\label{eqn.predictor}
	\begin{aligned} 
		& \int_{C_i^n} \theta_{i,k}^n(x,t) \partial_t q_{i}^n(x,t) \,\dx \dt  + 
		\int_{\Omega_i^n} \theta_{i,k}^n(x,t^n) q_{i}^n(x,t^n) \,\dx \\		
		&= \int_{\Omega_i^n} \theta_{i,k}^n(x,t^n) u_{i}^n(x) \,\dx
		- \int_{C_i^n} \theta_{i,k}^n(x,t) \partial_x f_{i}^n(x,t)\,\dx \dt  \qquad \text{for } k = 0,\dots, N_{st}-1.  
	\end{aligned} 
\end{equation}
By expanding $q_i^n$ and $u_i^n$ in \eqref{eqn.predictor}, see~\eqref{eq.qn} and \eqref{eq.un}, we obtain an implicit system of equations for the unknown coefficient vector $\hat q_i^n$. Here, the coefficient vector $\hat{f}_i^n$ is indeed a (generally nonlinear) function of $\hat{q}_{i}^n$, and so to approximate $\hat{q}_{i}^n$ we employ a fixed point Picard iteration, as detailed in~\cite{dumbser2008unified,hidalgo2011ader,busto2020high}. We can consider the vector $\hat q_i^{n,(0)}:=(\hat u_{i,0}^n, \dots, \hat u_{i,N}^n, 0, \dots, 0)$ as starting point for the iteration and remark that such fixed point procedure has already been proved to be convergent and to yield the desired order of 
accuracy, see~\cite{jackson2017eigenvalues,busto2020high,han2021dec} for more details.

Upon convergence, the predictor step provides us with a set of locally defined high order polynomials $q_i^n$ over each control volume $C_i^n$. These polynomials will serve as approximants of the solution in the interior of $C_i^n$ and will be used in the computation of the numerical fluxes at the interfaces.

\subsubsection{Corrector step}
\label{sec.corrector}
For each $i \in \{1,\dots,N_e\}$, we multiply the governing equation \eqref{eq.PDE} by a moving basis function $\psi_{i,k}^n$, $k = 0,\dots,N$, and integrate over the control volume $C_i^n$ to obtain
\begin{equation*}\label{eqn.intPDE}
	\int_{C_i^n} \psi_{i,k}^n(x,t) \left[ \partial_t Q(x,t) + \partial_x f(Q(x,t)) \right]\,\dx \dt = 0 \quad \text{for } k = 0,\dots, N.
\end{equation*}
Next, by integration by parts, we get 
\begin{equation*}
	\int_{\partial C_i^n} \psi_{i,k}^n (Q \, \hat n_{i,t}^n + f(Q) \, \hat n_{i,x}^n) \,\ds - \int_{C_i^n} \textcolor{black}{\left[\partial_t \psi_{i,k}^n Q + \partial_x \psi_{i,k}^n f(Q) \right]}\,\dx \dt = 0\quad \text{for } k = 0,\dots, N,
\end{equation*}
where $\mathbf{\hat n}_i^n = (\hat n_{i,x}^n,\hat n_{i,t}^n)$ denotes the outward pointing unit normal vector on the spacetime faces composing the boundary $\partial C_{i}^n$. We now decompose $\partial C_{i}^n$ into $\Omega_i^n$, $\Omega_i^{n+1}$ and the two lateral faces $\Sigma_{i\pm1/2}^n$, we recall that $u_i^n$ (resp. $u_i^{n+1}$) approximates $Q$ on $\Omega_i^n$ (resp. $\Omega_i^{n+1}$) and that the predictor $q_i^n$ approximates $Q$ inside the control volume $C_i^n$. Upon introducing a suitable numerical flux function $\mathcal{F} \colon \R\times \R \times S^1 \to \R$, we obtain
\begin{equation}\label{eqn.corrector}
	\begin{aligned} 
		&\int_{\Omega_i^{n+1}} \psi_{i,k}^n(x,t^{n+1}) u_{i}^{n+1}(x) \,\dx = 
		\int_{\Omega_{i}^n} \psi_{i,k}^n(x,t^n) u_{i}^{n}(x) \,\dx
		- \int_{\Sigma_{i-\frac12}^n} 
		\psi_{i,k}^n  \mathcal{F}(q_i^{n},q_{i-1}^{n}, \mathbf{\hat n}_i^n)\, \ds
		- \int_{\Sigma_{i+\frac12}^n} 
		\psi_{i,k}^n  \mathcal{F}(q_i^{n},q_{i+1}^{n}, \mathbf{\hat n}_i^n)\, \ds \\
		&+ \int_{C_i^n} 
		\partial_t \psi_{i,k}^n(x,t) q_i^n(x,t)  \,\dx \dt + \int_{C_i^n} \partial_x \psi_{i,k}^n(x,t) f(q_i^n(x,t)) \,\dx \dt \quad  \text{for }k = 0, \dots, N,
	\end{aligned} 
\end{equation}
where the coefficient vector $\hat u_i^{n+1}$ can be computed \emph{explicitly} upon knowing $\hat u_i^n$ and by integrating the remaining terms that depend only on the already computed predictors~$q_i^n$. The above formula is termed the \emph{corrector} step of the ADER-DG scheme.

The numerical flux function $\mathcal{F}$ is computed via an ALE Riemann solver applied to the inner and outer boundary-extrapolated data $q^-$ and $q^+$ at each boundary.
Here, the simplest choice consists in adopting a Rusanov-type~\cite{Rusanov:1961a} ALE flux. For $q^{-},q^{+} \in \R$ the interior and exterior extrapolated values, and for an outgoing normal vector $\mathbf{\hat n} = (\hat n_x, \hat n_t) \in S^1$,  we set
\begin{equation}\label{eq.rusanov} 
	\mathcal{F}(q^{-},q^{+},\mathbf{\hat n})  =  
	\frac{1}{2} \left( {f}(q^{+}) + {f}(q^{-})  \right) \hat n_x
	+ \frac{1}{2} \left( q^{+} + q^{-}  \right) \hat n_t
	- \frac{1}{2} s_{\max} \left( q^{+} - q^{-} \right),  
\end{equation} 
where, in full generality, $s_{\max}$ is the maximum of the spectral radii of the ALE Jacobian matrix w.r.t. the normal direction in space evaluated at $q^-$ and $q^+$, which here simplifies to
\[
s_{\max} = \max\{ |f'(q^-)  - v|, |f'(q^+)  - v| \} \cdot |\hat n_x|,
\]
where $v = -\hat n_t/|\hat n_x|$ is the local grid velocity. For the general expression for systems and in higher dimension we refer to \cite{boscheri2013arbitrary}. 


\subsection{Implicit ALE ADER-DG method}
Opposite to the explicit method, the implicit ALE ADER-DG method directly seeks a \emph{global} approximant $q^n$ by means of simultaneously computing all approximants $q_i^n$ over each control volume $C_i^n$. For each $i \in \{1,\dots,N_e\}$, we multiply the governing equation \eqref{eq.PDE} by a spacetime basis function $\theta_{i,k}^n$, $k = 0,\dots,N_{st}-1$, and integrate over the control volume $C_i^n$ to obtain
\begin{equation*}
	\int_{C_i^n} \theta_{i,k}^n(x,t) \left[ \partial_t Q(x,t) + \partial_x f(Q(x,t)) \right]\,\dx \dt = 0 \quad \text{for } k = 0,\dots, N_{st}-1.
\end{equation*}
Next, by integration by parts, we get 
\begin{equation*}
	\int_{\partial C_i^n} \theta_{i,k}^n (f(Q), Q) \cdot \mathbf{\hat n}_i^n \,\ds - \int_{C_i^n} \textcolor{black}{\left[ \partial_t \theta_{i,k}^n Q + \partial_x \theta_{i,k}^n f(Q) \right]} \,\dx \dt = 0\quad \text{for } k = 0,\dots, N_{st}-1,
\end{equation*}
where $\mathbf{\hat n}_i^n = (\hat n_{i,x}^n,\hat n_{i,t}^n)$ denotes the outward pointing spacetime unit normal. Now, on each control volume $C_i^n$ we substitute $Q$ by an approximant $q_i^n$ in the form \eqref{eq.qn}, we decompose the boundary integral into four integrals (one over each of the four faces), we introduce a suitable numerical flux $\mathcal{F}$ and require $q_i^n(x,t^n) = u_i^n(x)$ for $x \in \Omega_i^n$, to get
\begin{equation}\label{eq.implicit}
	\begin{aligned} 
		&\int_{\Omega_i^{n+1}} \theta_{i,k}^n(x,t^{n+1}) q_i^n(x,t^{n+1}) \,\dx 
		- \int_{C_i^n} \partial_t \theta_{i,k}^n(x,t) q_i^n(x,t) \,\dx \dt
		- \int_{C_i^n} \partial_x \theta_{i,k}^n(x,t) f_i^n(x,t) \,\dx \dt \\
		& + \int_{\Sigma_{i-\frac12}^n} 
		\theta_{i,k}^n  \mathcal{F}(q_i^{n},q_{i-1}^{n}, \mathbf{\hat n}_i^n)\, \ds
		+ \int_{\Sigma_{i+\frac12}^n} 
		\theta_{i,k}^n  \mathcal{F}(q_i^{n},q_{i+1}^{n}, \mathbf{\hat n}_i^n)\, \ds
		=
		\int_{\Omega_{i}^n}	\theta_{i,k}^n(x,t^n)  u_i^n(x) \,\dx \\
		&\text{for } i = 1, \dots, N_e \text{ and } k = 0, \dots, N_{st}-1,
	\end{aligned} 
\end{equation}
where $f_i^n$ is defined as in \eqref{eq.fhat}. The set of equations \eqref{eq.implicit} implicitly defines a generally nonlinear system for the coefficients $\{\hat{q}_{i,\ell}^n\}_{i,\ell}$, which we solve by means of a Newton iteration coupled with a GMRES algorithm. Once each $q_i^n$ has been computed, the next approximant $u_i^{n+1}$ can be recovered through an $L^2$ projection of $q_i^n(\cdot,t^{n+1})$ with respect to the basis $\{\phi_{i,\ell}^{n+1}\}_\ell$. Note that here, in parallel to the explicit case, we have termed the implicit version of the scheme as implicit ADER-DG, while in other contexts the same method may be referred to as implicit spacetime DG.

\section{Stability of ALE ADER-DG methods on classical geometries}
\label{sec.stability-classical}

We move now to the focus of this paper. We are interested in studying the stability of the explicit ADER-DG method introduced in Section~\ref{sec.method-description} by means of the well known von Neumann stability analysis technique \cite[Sections 9.6 and 10.5]{leveque2007finite}. We consider the linear advection equation (LAE) for a constant advection velocity $a > 0$, so that $f(Q) = aQ$. We restrict ourselves to the domain $\Omega=[0,2\pi]$ and consider a uniform grid of $N_e$ elements of size $\Delta x > 0$ and a positive $\Delta t$. For a given wave number $\kappa\in \mathbb \N$, we consider the periodic initial condition $u_0(x) = e^{\icomplex\kappa x}$ and close the system with periodic boundary conditions. Since the initial condition satisfies $u_0(x - \Delta x) = e^{-\icomplex\kappa \Delta x}u_0(x)$ for every $x \in \R$, it can be easily proved that the coefficients generated by an ADER-DG method on a fixed grid satisfy $\hat{u}_{i-1}^n = e^{-\icomplex\kappa \Delta x} \hat{u}_{i}^n$ for any $i, n$ (with the appropriate modifications at the boundaries).

The purpose of this section is then to study some necessary $\CFL$-type bounds for the ADER-DG scheme applied to the periodic LAE. These bounds are required to maintain a bounded norm for the coefficient vectors $\{\hat{u}_i^n\}_{i}$ over time, and they provide an informed guess for time step selection in complex scenarios where a direct stability analysis is less feasible.

\subsection{Stability of the explicit method}
To describe the explicit ADER-DG as applied to this setting, we first observe that we can drop the dependency of basis functions on the individual elements, since each $\Omega_i^n$ can be obtained (up to translation) from the interval $[0,1]$ rescaled by $\Delta x$ and each control volume $C_i^n$ from the square $[0,1] \times [0,1]$ rescaled by $\Delta x$ in space and by $\Delta t$ in time. We conveniently define the reference basis functions
\[
\begin{aligned}
	\phi_{\ell} \colon [0,1] &\to \R, \, \xi \mapsto (\xi - 1/2)^\ell
	& \text{for } \ell &= 0,\dots,N, \\
	\psi_{\ell} \colon [0,1]\times[0,1] &\to \R, \, (\xi,\tau) \mapsto \phi_{\ell}(\xi)
	& \text{for } \ell &= 0,\dots,N,
\end{aligned}
\]
and
\[
\begin{aligned}
	&\theta_{\ell} \colon [0,1] \times [0,1] \to \R, \, (\xi,\tau) \mapsto \left(\xi - 1/2 \right)^{\ell_1} \, \tau^{\ell_2}
	\quad \text{for } \ell = \ell_1 + \ell_2(2N + 3 - \ell_2)/2, \,\, 0 \leq \ell_1 + \ell_2 \leq N.
	\nonumber
\end{aligned}
\]
Since $f(Q) = aQ$, each predictor step in \eqref{eqn.predictor} simplifies to
\begin{equation}\label{eq:laepredictor}
	\hat{q}_i^n = \left(K_\tau^{st} + K_0^{st} + \frac{a\Delta t}{\Delta x}K_\xi^{st}\right)^{-1} M_0 \hat{u}_i^n,
\end{equation}
where the three $N_{st}\times N_{st}$ matrices $K_\tau^{st}, K_0^{st}, K_\xi^{st}$ and the $N_{st} \times (N+1)$ matrix $M_0$ are defined as
\begin{equation}\label{eq.matricespredictor}
	\begin{aligned}
		[ K_\tau^{st} ]_{k, \ell} &= \int_0^1\int_0^1 \theta_{k}(x,t) \partial_t \theta_\ell(x,t) \,\dx\dt, \qquad& [K_\xi^{st} ]_{k, \ell} &= \int_0^1\int_0^1 \theta_{k}(x,t) \partial_x \theta_\ell(x,t) \,\dx\dt,
		\\
		[ K_0^{st} ]_{k, \ell} &= \int_0^1 \theta_{k}(x,0) \theta_\ell(x,0) \,\dx, \qquad& [ M_0 ]_{k, \ell} &= \int_0^1 \theta_{k}(x,0) \phi_\ell(x) \,\dx.
		\\
	\end{aligned}
\end{equation}

By taking now into account that the Rusanov-type flux \eqref{eq.rusanov} for the LAE is a simple upwind scheme and recalling $a > 0$, the subsequent corrector step in \eqref{eqn.corrector} simplifies to
\begin{equation}\label{eq:laecorrector}
	\hat{u}_i^{n+1} = \hat{u}_i^{n} + \frac{a\Delta t}{\Delta x} M^{-1} (K_\xi\hat{q}_i^n - F_{\text{-}}^{r}\hat{q}_i^n + F_{\text{-}}^{l}\hat{q}_{i-1}^n),
\end{equation}
where the three $(N+1) \times N_{st}$ matrices $K_\xi, F_{\text{-}}^{r}, F_{\text{-}}^{l}$ and the $(N+1) \times (N+1)$ mass matrix $M$ are defined as
\begin{equation}\label{eq.matricescorrector}
	\begin{aligned}
		[ M ]_{k, \ell} &= \int_0^1 \phi_{k}(x) \phi_\ell(x) \,\dx, \qquad& [ K_\xi ]_{k, \ell} &= \int_0^1\int_0^1 \partial_x \psi_{k}(x,t) \theta_\ell(x,t) \,\dx\dt,
		\\
		[ F_{\text{-}}^{r} ]_{k, \ell} &= \int_0^1 \psi_{k}(1,t) \theta_\ell(1,t) \,\dt, \qquad& [ F_{\text{-}}^{l} ]_{k, \ell} &= \int_0^1 \psi_{k}(0,t) \theta_\ell(1,t) \,\dt.
		\\
	\end{aligned}
\end{equation}
Since $\hat{u}_{i-1}^n = e^{-\icomplex\kappa \Delta x} \hat{u}_{i}^n$, we have from \eqref{eq:laepredictor} that $\hat{q}_{i-1}^n = e^{-\icomplex\kappa \Delta x} \hat{q}_{i}^n$. Hence, the update in \eqref{eq:laecorrector} reduces to
\begin{equation*}
	\begin{aligned}
		\hat{u}_i^{n+1} &= \hat{u}_i^{n} + \frac{a\Delta t}{\Delta x}M^{-1} \left(K_\xi - F_{\text{-}}^{r} + e^{-\icomplex\kappa \Delta x}F_{\text{-}}^{l}\right)\hat{q}_i^n \\
		&= \hat{u}_i^{n} + \frac{a\Delta t}{\Delta x} M^{-1} \left(K_\xi - F_{\text{-}}^{r} + e^{-\icomplex\kappa \Delta x}F_{\text{-}}^{l}\right)\left(K_\tau^{st} + K_0^{st} + \frac{a\Delta t}{\Delta x}K_\xi^{st}\right)^{-1} M_0 \hat{u}_i^n \\
		&= A_N(\CFL, \theta) \cdot \hat{u}_i^n,
	\end{aligned}
\end{equation*}
with an amplification matrix $A_N(\CFL, \theta)\in \mathbb C^{(N+1)\times (N+1)}$ dependent only on $\theta = \kappa \Delta x$ and $\CFL = a\Delta t/\Delta x$. We observe that every matrix in \eqref{eq.matricespredictor} and \eqref{eq.matricescorrector} can be computed analytically; consequently, the same holds for each amplification matrix $A_N$.
For a matrix $A$, let us denote by $\rho(A)$ its spectral radius. Then, for a given $\CFL$ number, a necessary condition for the update formula for $\hat{u}_i^n$ to be stable, regardless of the value of $\Delta x$, is
\begin{equation}\label{eq.stabilitycondition}
	\rho(A_N(\CFL, \theta)) \leq 1   \quad \text{for every } \theta \in [0,\pi].
\end{equation}
In \eqref{eq.stabilitycondition}, the phase angle can be restricted to $[0,\pi]$ since, by construction,  $\rho(A_N(\CFL, \theta)) = \rho(A_N(\CFL, -\theta))$ for every $\CFL \geq 0$ and $\theta \in [0,\pi]$.

\medskip
\noindent
\emph{Numerical setting for studying stability conditions.}

\begin{table}[t]
	\centering
	\caption{Increment indicator $\Delta_N^{m \to m+1}$ for varying degrees $N$ and for $m = 1, \dots, 11$. At each step, we evaluate how the discrete amplification factor is affected upon doubling the subdivision of the interval $[0,\pi]$. The low increments motivate the sampling decision $m = 6$.}
	\label{tab:results_comparison}
		\begin{tabular}{l ccccccccc}
			\toprule
			{} & {$N = 1$} & {$N = 2$} & {$N = 3$} & {$N = 4$} & {$N = 5$} & {$N = 6$} & {$N = 7$} & {$N = 8$} & {$N = 9$} \\
			\midrule
			$\Delta^{1 \to 2}$ & 0 & 0 & 0 & 4.11\es{-06} & 9.58\es{-07} & 3.54\es{-05} & 1.13\es{-05} & 8.13\es{-06} & 6.72\es{-09} \\
			$\Delta^{2 \to 3}$ & 0 & 0 & 0 & 3.02\es{-08} & 9.91\es{-07} & 3.73\es{-06} & 7.02\es{-06} & 9.55\es{-06} & 6.79\es{-14} \\
			$\Delta^{3 \to 4}$ & 0 & 0 & 0 & 3.34\es{-08} & 6.28\es{-07} & 1.82\es{-06} & 2.96\es{-06} & 1.46\es{-06} & 1.22\es{-13} \\
			$\Delta^{4 \to 5}$ & 0 & 0 & 0 & 3.29\es{-08} & 2.70\es{-07} & 6.53\es{-07} & 9.05\es{-07} & 7.10\es{-07} & 3.51\es{-13} \\
			$\Delta^{5 \to 6}$ & 0 & 0 & 0 & 4.33\es{-08} & 1.46\es{-07} & 1.98\es{-07} & 2.05\es{-07} & 1.59\es{-07} & 2.29\es{-12} \\
			$\Delta^{6 \to 7}$ & 0 & 0 & 4.44\es{-16} & 7.84\es{-09} & 3.50\es{-08} & 5.52\es{-08} & 5.37\es{-08} & 3.72\es{-08} & 1.55\es{-11} \\
			$\Delta^{7 \to 8}$ & 0 & 0 & 4.44\es{-16} & 2.30\es{-09} & 9.38\es{-09} & 1.31\es{-08} & 1.47\es{-08} & 1.07\es{-08} & 2.86\es{-12} \\
			$\Delta^{8 \to 9}$ & 0 & 0 & 4.44\es{-16} & 7.10\es{-10} & 2.41\es{-09} & 3.22\es{-09} & 3.42\es{-09} & 2.27\es{-09} & 7.95\es{-13} \\
			$\Delta^{9 \to 10}$ & 0 & 0 & 6.66\es{-16} & 1.80\es{-10} & 6.10\es{-10} & 8.92\es{-10} & 8.59\es{-10} & 6.34\es{-10} & 3.28\es{-13} \\
			$\Delta^{10 \to 11}$ & 0 & 0 & 4.44\es{-16} & 4.09\es{-11} & 1.48\es{-10} & 2.07\es{-10} & 2.33\es{-10} & 1.67\es{-10} & 8.93\es{-14} \\
			$\Delta^{11 \to 12}$ & 0 & 0 & 4.44\es{-16} & 1.13\es{-11} & 3.42\es{-11} & 5.42\es{-11} & 5.26\es{-11} & 3.87\es{-11} & 8.99\es{-14} \\
			\bottomrule
		\end{tabular}
\end{table}

\noindent
As we know that a finite volume scheme requires $\CFL \leq 1$, we consider a uniform subdivision of the $\CFL$ domain $[0, 1]$ into $10^4$ elements, i.e., we check condition \eqref{eq.stabilitycondition} over the discrete sample set
$
\{ c = k10^{-4} \mid k \in \N \} \cap [0,1].
$
Without loss of generality, we can focus our study on the set of $\CFL$ values for which the amplification factor related to constant states is no bigger than $1.1$. Hence, we define the set
\begin{equation}\label{eq.CFLrange}
	C_N := \left\{ c = k \cdot 10^{-4} \mid k \in \N, \,\rho(A_N(c, 0)) \leq 1.1 \right\} \quad \text{for each } N = 1,\dots,9.
\end{equation}
To numerically evaluate \eqref{eq.stabilitycondition}, we consider a finite set of phase angles $\theta$. We subdivide the interval $[0,\pi]$ into $2^m$ uniform subintervals, for a fixed $m \in \N$, and check condition \eqref{eq.stabilitycondition} at the corresponding $2^m+1$ nodes, obtaining the discrete amplification factor
\begin{equation}\label{eq.discretestabilitycondition}
	\rho_{N,m}(c) = \max_{k = 1,\dots,2^{m}+1} \rho\left(A_N\left(c, \frac{k-1}{2^{m}}\pi\right)\right) \quad \text{for } c \in [0,1].
\end{equation}
Since by construction $\rho_{N,m} \leq \rho_{N,m+1}$, to study the effect of the choice of $m$ on the discrete amplification factor, we can define the non-negative \emph{increment indicator}
\begin{equation}\label{eq.nextgridindicator}
	\Delta_N^{m \to m+1} = \max_{c \in C_N} \left\{ \rho_{N,m+1}(c) - \rho_{N,m}(c) \right\},
\end{equation}
that measures the variation of the discrete maximum when doubling the evaluation points. A thorough study of this increment indicator is reported in Table \ref{tab:results_comparison}. We can observe how the discrete amplification factor stabilises very quickly as we increase $m$, independently of the degree $N$. For $N = 1,2,3$, the maximum seems to be achieved at $0$ and/or $\pi$ (as the increment indicator is always zero), while for $N \geq 4$ the discrete amplification factor is not increasing more than $10^{-7}$ already for $m>6$. In particular, if we take into account that the relevant amplification factors in Figure~\ref{fig:stabilityADERDGexplicit} are around $1+10^{-4}$, in the sequel, unless differently specified, we will always compute discrete maxima by subdividing $[0,\pi]$ into $2^6$ uniform elements, i.e., we fix $m = 6$.


\medskip
\noindent
\emph{Identification of relevant $\CFL$ bounds.}

\begin{table}[tb]
	\centering
	\caption{Relevant $\CFL$ bounds for varying polynomial degrees $N = 1,\dots, 9$. First line: commonly used empirical $\CFL$ bounds, as reported in \cite{Gaburro2021}. Second line: largest tested $\CFL$ value such that $\rho_{N,6}(\CFL) \leq 1+10^{-3}$. Third line: largest tested $\CFL$ value such that $\rho_{N,11}(\CFL) \leq 1+10^{-10}$.}
	\label{PNPN}
		\begin{tabular}{l *{9}{l}}
			\toprule
			& \multicolumn{1}{l}{$N = 1$} & \multicolumn{1}{l}{$N = 2$} & \multicolumn{1}{l}{$N = 3$} & \multicolumn{1}{l}{$N = 4$} & \multicolumn{1}{l}{$N = 5$} & \multicolumn{1}{l}{$N = 6$} & \multicolumn{1}{l}{$N = 7$} & \multicolumn{1}{l}{$N = 8$} & \multicolumn{1}{l}{$N = 9$} \\
			\midrule
			$\CFL_{empirical}$    & 0.33\phantom{00} & 0.17\phantom{00} & 0.10\phantom{00} & 0.069\phantom{0} & 0.045\phantom{0} & 0.038\phantom{0} & 0.03\phantom{00} & 0.02\phantom{00} & 0.015\phantom{0} \\
			$\CFL_{\rho \leq 1+1\es{-3}}$    & 0.3334           & 0.1709           & 0.1039           & 0.0698           & 0.0501           & 0.0377           & 0.0293           & 0.0235           & 0.0192           \\
			$\CFL_{\rho \leq 1+1\es{-10}}$   & 0.3333           & 0.1708           & 0.1039           & 9.23\es{-4} & 5.07\es{-4} & 7.76\es{-4} & 3.67\es{-3} & 9.89\es{-3} & 1.58\es{-3} \\
			\bottomrule
		\end{tabular}
\end{table}

\noindent
For each $N = 1,\dots,9$, we compute the discrete amplification factor $\rho_{N,6}$. We observe that the explicit ADER-DG scheme always preserves constant states: indeed, for $\theta = 0$ and $\hat{u}_i^n = (1,0,\dots,0)$, we have $A_N(c, 0) \cdot \hat{u}_i^n = \hat{u}_i^n$, independently of $c \in [0,1]$. Hence, $\rho_{N,m}(c) \geq 1$ for every $c$ and every $m$, in particular $\rho_{N,6} \geq 1$. In view of this, we report in Figure~\ref{fig:stabilityADERDGexplicit} the graph of $\rho_{N,6}-1$. Furthermore, the second row of Table~\ref{PNPN} identifies the highest $c \in C_N$ such that $\rho_{N,6}(c) \leq 1+10^{-3}$. For comparison, the first row recalls the $\CFL$ bounds widely used in the community (see, e.g., \cite{Gaburro2021, chiocchetti2021high, ricardo2025scalable}), which have generally been determined empirically.

For $N \leq 3$, the empirical $\CFL$ bounds and the $\CFL$ bounds determined by our linear stability analysis coincide, in accordance with the analysis carried out in~\cite{dumbser2008unified}. On the other hand, for $4 \leq N \leq 9$, the method turns out to be unstable for commonly used $\CFL$ bounds. By comparing the first and second rows in Table~\ref{PNPN}, we observe that empirical $\CFL$ bounds just correspond to the rightmost jump in each stability plot, where the stability condition \eqref{eq.stabilitycondition}, even if only slightly, is violated. Hence, commonly used $\CFL$ values do not correspond to true stability limits. However, this does not mean that they cannot be used in applications as an informed guess for the selection of the time step. Indeed, in complex test cases, the small amplification factor due to instability is mitigated when combined with other discretization techniques such as numerical viscosity, Riemann solvers, damping factors, and limiters. As a result, authors usually adopt these values as their reference maximum $\CFL$ number when running numerical experiments.

A sharp von Neumann limit can be recovered for $N \geq 4$ only for much smaller values of the $\CFL$ number, as shown in Figure~\ref{fig:stabilityADERDGexplicitzoom}, where we pick a uniform $\CFL$ sampling in the log-domain (starting from $10^{-12}$ with a log-step of $0.005$). We report in the third and last row of Table~\ref{PNPN} the largest tested $\CFL$ value $c$ such that $\rho_{N,11}(c) \leq 1+10^{-10}$.

\begin{figure}[]
	\centering
	\includegraphics[width=0.9\linewidth]{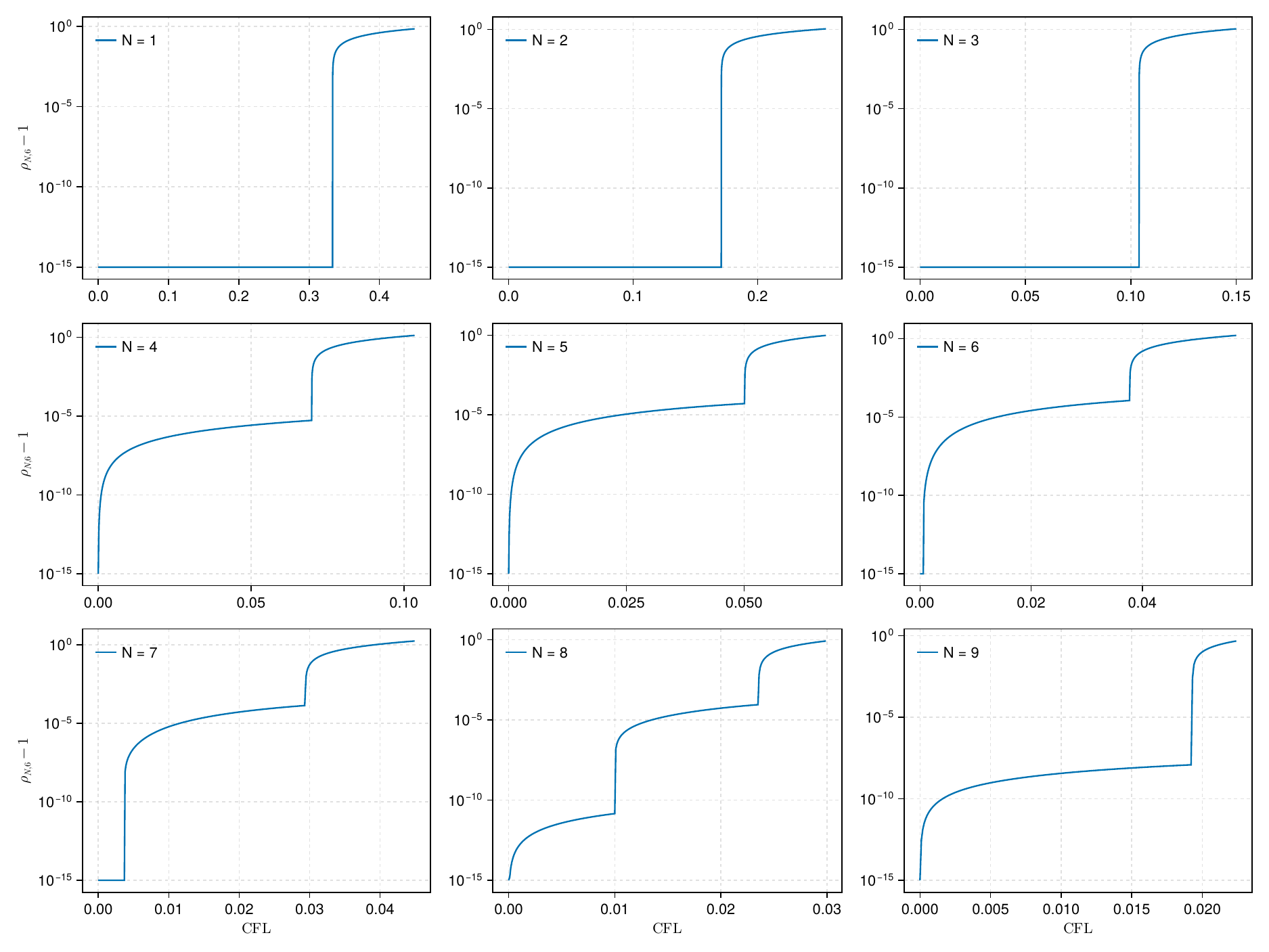}
	\caption{Semi-logarithmic plot of the function $\CFL \mapsto \rho_{N,6}(\CFL) - 1$, where the discrete maximal spectral radius is computed over $2^6+1$ equispaced points over $[0, \pi]$. The necessary stability condition for the explicit ADER-DG method requires this function to vanish, see \eqref{eq.stabilitycondition}.}
	\label{fig:stabilityADERDGexplicit}
	
	\vspace{1em}
	
	\includegraphics[width=0.9\linewidth]{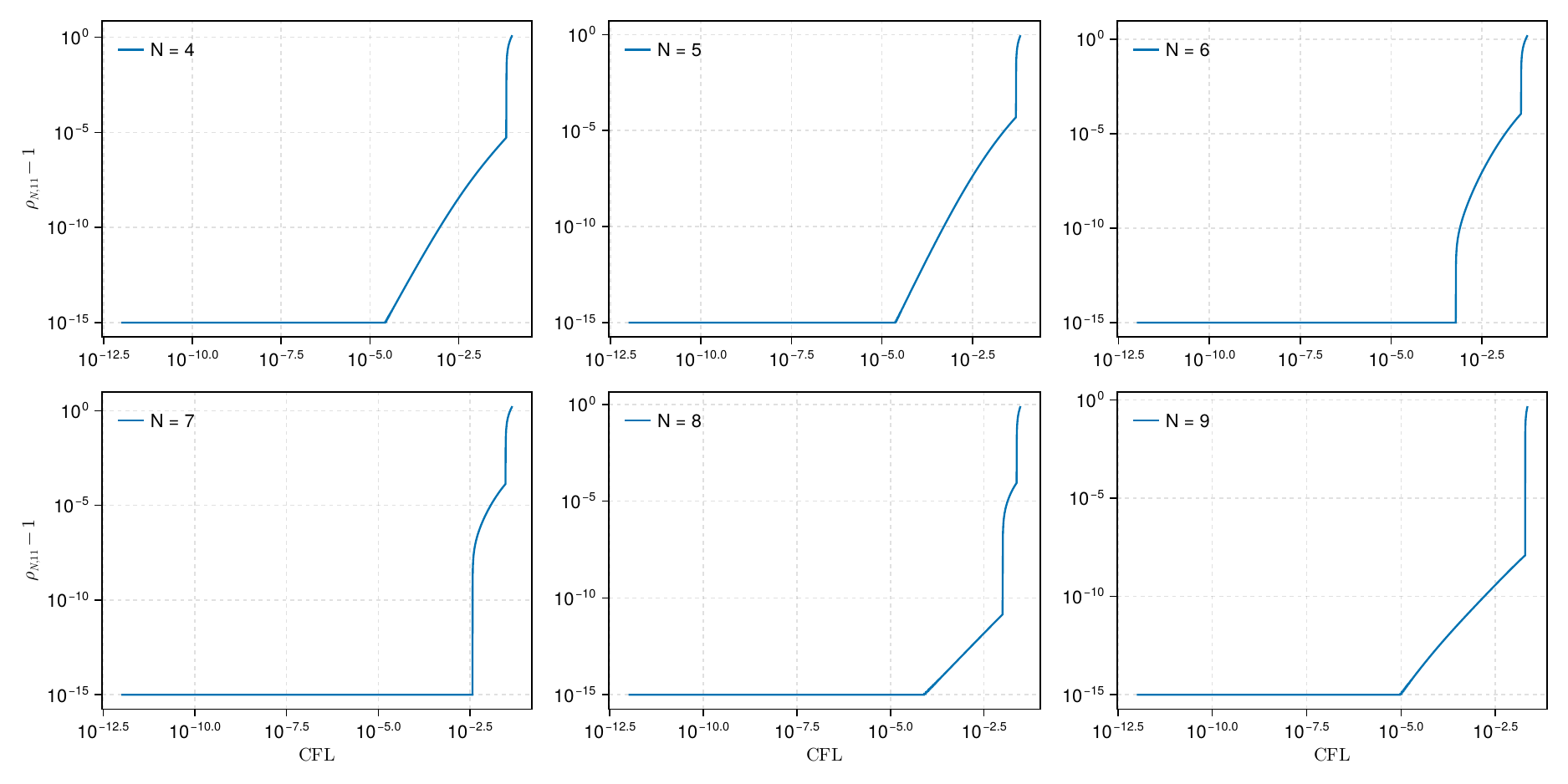}
	\caption{Logarithmic plot of the function $\CFL \mapsto \rho_{N,11}(\CFL) - 1$, where the maximum spectral radius is computed over $2^{11}+1$ equispaced points over $[0, \pi]$. Two clear corner points identify the stability threshold that often corresponds to very low $\CFL$ values, while the larger one is the typically used empirical $\CFL$  bound.}
	\label{fig:stabilityADERDGexplicitzoom}
\end{figure}
 
\begin{remark}\label{rem:ALE}
	The analysis extends to the case of a mesh moving with uniform velocity $v \in \R$. In this context, the stability condition turns into
	\[
	|a - v|\frac{\Delta t}{\Delta x} \leq \CFL_{\textup{max}}.
	\]
	Assuming $a > 0$, by using a value $\tilde{c} \leq \CFL_{\textup{max}}$ and selecting $\Delta t = \tilde{c} \Delta x/a$, the condition above rewrites as
	\[
	a\left(1-\frac{\CFL_{\textup{max}}}{\tilde{c}} \right) \leq v \leq a\left(1+\frac{\CFL_{\textup{max}}}{\tilde{c}} \right),
	\]
	which provides a bound on the admissible velocities of the grid. For example, if $\tilde{c} = \CFL_{\textup{max}}$ one obtains $0 \leq v \leq 2a$, while if $\tilde{c} = 0.5\cdot \CFL_{\textup{max}}$ the condition becomes $-a \leq v \leq 3a$, allowing more freedom in the mesh velocities. This suggests the time step selection rule
	\[
	\Delta t^n \leq \tilde{c} \min_{i = 1,\dots,N_e} \frac{\Delta x_i^n}{\max\{|f'(u_i^n(x_{i-1/2}^n))|, |f'(u_i^n(x_{i+1/2}^n))|\}},
	\]
	without explicitly incorporating the mesh velocity, but accounting for it implicitly by selecting a suitable $\tilde{c}$.
\end{remark}

\begin{remark}
\textcolor{black}{Recently, \cite{popov2026} independently derived the $\CFL$  bounds in the second row of Table~\ref{PNPN} through a thorough analytical treatment of the amplification matrix $A_N$, under a suitable ansatz on the behaviour of its eigenvalues. Despite employing spacetime tensor basis functions, their von Neumann stability analysis retrieves identical bounds and yields the same conclusions regarding the stability of the explicit ADER-DG scheme.}
\end{remark}

\subsection{Numerical convergence  order of the explicit method}
\label{ssec.convergence-classical-explicit}

The explicit ADER-DG method is expected to have order of convergence  $N+1$ for a given polynomial degree $N$. We verify this property by evolving a linear advection equation with unit velocity on the domain $[-6,6]$, with periodic boundary conditions and starting from $Q(x,0) = e^{-x^2}$. We vary the number of elements $N_e$ and we consider
a $\CFL$ at $0.9 \cdot \CFL_{\textup{max}}$, where $\CFL_{\textup{max}}$ is taken from the second row of Table~\ref{PNPN}. In Figure~\ref{fig:aderconsistencyexplicit}, we report, in logarithmic scale, the $L^2$ norm of the error with respect to the exact solution at time $T=12$, i.e., after a full loop. Each convergence  order is correctly achieved as expected.

\begin{figure}[t]
	\centering
	\begin{subfigure}{0.49\linewidth}
		\includegraphics[width=\linewidth]{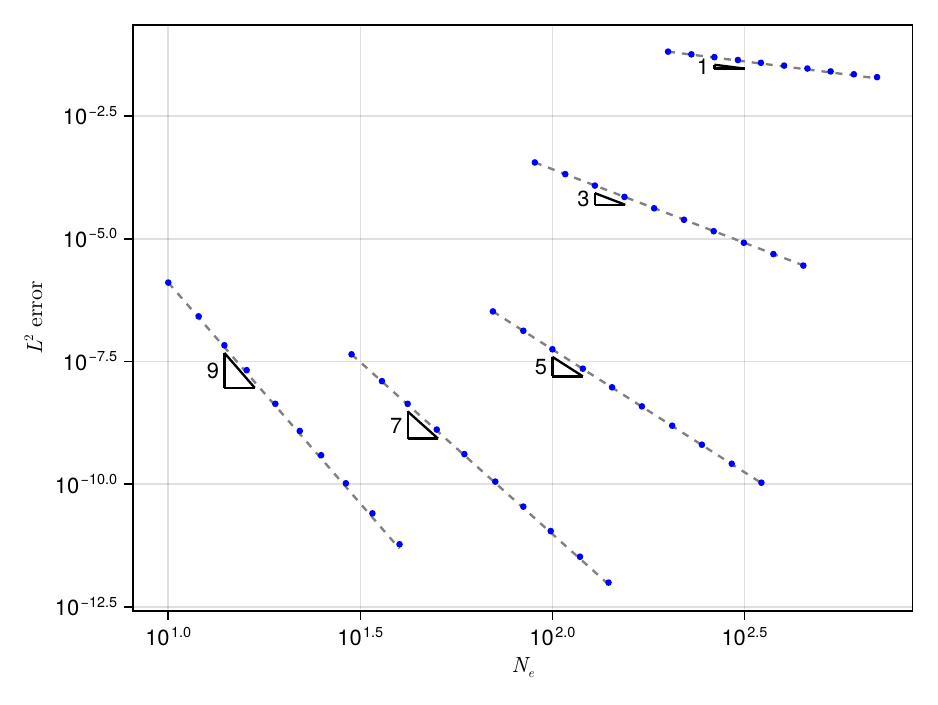}
	\end{subfigure}
	\hfill
	\begin{subfigure}{0.49\linewidth}
		\includegraphics[width=\linewidth]{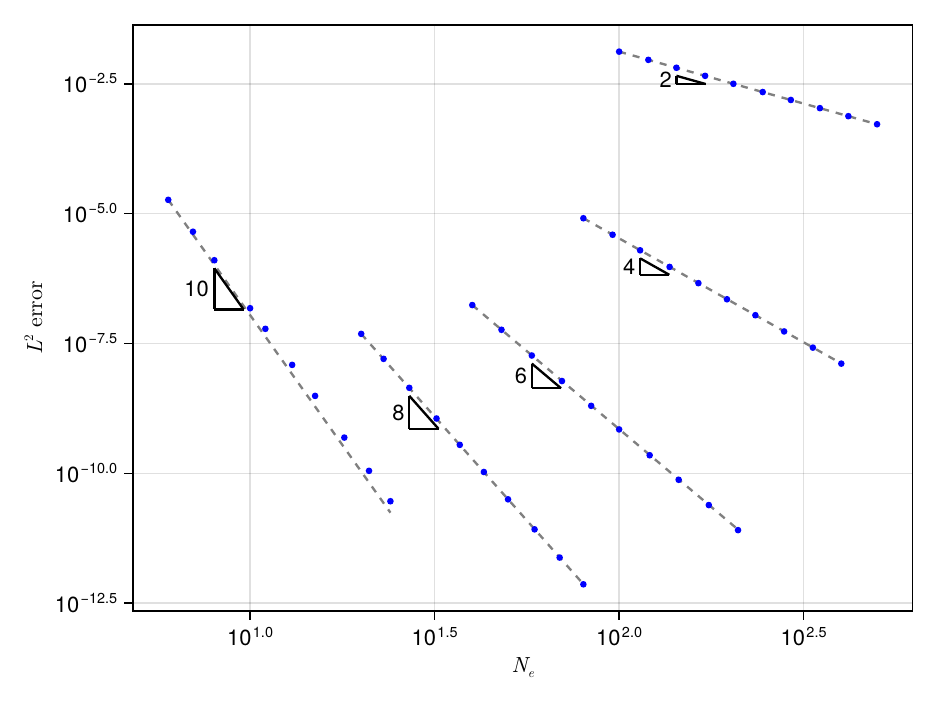}
	\end{subfigure}
	\caption{Convergence  order of the explicit ADER-DG method: left odd orders for $N = 0,2,4,6,8$, right even orders for $N = 1,3,5,7,9$.}
	\label{fig:aderconsistencyexplicit}
\end{figure}

\subsection{Stability of the implicit method}

The stability analysis for the implicit ADER-DG method can be developed following the same steps used for the explicit counterpart. Again, we have to study the spectral radius of a suitable amplification matrix $A_N(\CFL, \theta)$, which can be computed exactly. We evaluate condition \eqref{eq.stabilitycondition} by means of computing $\rho_{N,6}$, see \eqref{eq.discretestabilitycondition}, on the discrete set
$
C = \{c = k \cdot 10^{-2} \mid k \in \mathbb{N} \} \cap [0, 10]
$.
We obtain the following:
\begin{center}
	\resizebox{\textwidth}{!}{%
		\begin{tabular}{c|ccccccccc}
			& $N=1$ & $N=2$ & $N=3$ & $N=4$ & $N=5$ & $N=6$ & $N=7$ & $N=8$ & $N=9$ \\ \hline \\[-10pt]
			$\max\limits_{c \in C} \rho_{N,6}(c) - 1$ & 0 & 8.88\es{-16} & 2.89\es{-15} & 3.55\es{-15} & 4.00\es{-15} & 2.95\es{-13} & 2.38\es{-14} & 2.52\es{-13}& 2.03\es{-12} \\ 
		\end{tabular}.%
	}
\end{center}
No clear $\CFL$-type bounds can be identified, suggesting unconditional stability.
This numerical check is coherent with the $L^2$ stability of the scheme provided in \cite{dumbser2016space}. \textcolor{black}{In the specific von Neumann setting, we can actually prove stability of the scheme for a slightly different choice of spacetime basis functions, namely tensor basis functions. This is possible because, in this new framework, the method can be cast as a Runge–Kutta method whose stability was analysed in \cite{offner2025analysis}. While the choice of spacetime basis functions is different, the scheme is expected to behave similarly (see also the recent results in \cite{popov2026}, where explicit stability bounds for the explicit ADER-DG method were obtained in a spacetime tensor basis framework).} We can prove the following.

\begin{theorem}
	Consider the linear advection equation
	\[
	\partial_t u(x,t) + \partial_x u(x,t) = 0 \quad \text{for } x\in(0,2\pi), t>0,
	\]
	coupled with periodic boundary conditions and initial datum $u(x,0) = e^{\icomplex\kappa x}$, $\kappa \in \N$. Consider over $[0,2\pi]$ a uniform Eulerian grid of $N_e > 0$ elements of size $\Delta x > 0$. Consider the Rusanov-type numerical flux \eqref{eq.rusanov}. Then, when applied to this problem, the implicit ADER-DG method as described in \eqref{eq.implicit}, \textcolor{black}{with spacetime tensor basis functions}, is stable for any $\Delta t > 0$.
\end{theorem}
\begin{proof}
	In the uniform Eulerian case, we have that $\Omega_i^{n+1}=\Omega_i^{n}=\Omega_i$, each spatial element has measure $\Delta x > 0$ and $C_i^n = \Omega_i \times [t^n,t^{n+1}]$. In space, over each element $\Omega_i$ we can use Lagrange basis functions $\{\psi^s_{i,k_s}\}_{k_s=0}^{N}$ defined on Gauss--Legendre points, and similarly in time we use Lagrange basis functions $\{\psi^t_{n,k_t}\}_{k_t=0}^{N}$ defined on Gauss--Legendre points over $[t^n, t^{n+1}]$. Indeed, this choice is coherent with the one made in Section~\ref{sec.method-description} because it has been shown that any polynomial choice with exact quadrature leads to an equivalent method \cite{veiga2024improving}. Spacetime basis functions are then defined in a tensor fashion as $\theta_{i,(k_s, k_t)}^n(x,t) = \psi^s_{i,k_s}(x)\psi^t_{n,k_t}(t)$, for indices $k_s, k_t \in \{0,\dots,N\}$. In particular, approximants take now the form
	\[
	q^n_{i}(x,t) = \sum_{k_s, k_t = 0}^N\theta_{i,(k_s, k_t)}^n(x,t)\hat{q}_{i,(k_s, k_t)}^n \quad \text{and} \quad u^n_{i}(x) = \sum_{k_s = 0}^N\psi^s_{i,k_s}(x)\hat{u}^n_{i,k_s}
	\]
	for given coefficients $\{\hat{q}_{i,(k_s, k_t)}^n\}_{k_s, k_t} \subset \R$ and $\{\hat{u}^n_{i,k_s}\}_{k_s} \subset \R$. In this setting, the implicit ADER-DG method in \eqref{eq.implicit} can be written, using integration by parts in time, as 
	\[
	\begin{aligned}
			&\int_{\Omega_i} \theta_{i,(k_s,k_t)}^n(x,t^{n+1}) q_i^n(x,t^{n+1}) \,\dx -\int_{\Omega_i} \theta_{i,(k_s,k_t)}^n(x,t^{n}) u_i^n(x) \,\dx -\int_{C^n_i} \partial_t  \theta_{i,(k_s,k_t)}^n(x,t) q_i^n(x,t) \,\dx\dt \\
			&-\int_{C^n_i} \partial_x  \theta_{i,(k_s,k_t)}^n(x,t) q_i^n(x,t) \,\dx\dt + \int_{t^n}^{t^{n+1}} \left( \theta_{i,(k_s,k_t)}^n(x_{i+\frac12},t) q_i^n(x_{i+\frac12},t) -\theta_{i,(k_s,k_t)}^n(x_{i-\frac12},t) q_{i-1}^n(x_{i-\frac12},t) \right) \,\dt=0
		\end{aligned}
	\]
	for $i = 1, \dots, N_e$ and $k_s, k_t = 0, \dots, N$. Now, by using the von Neumann ansatz that $q_{i-1}^n = e^{-\icomplex \theta}q_i^n$, $\theta = \kappa \Delta x$, and omitting the cell indexes $i,n$, we define
	\[
	\begin{aligned}
			&[M^s]_{k_s,j_s} := \int_{\Omega_i} \psi^s_{k_s} \psi^s_{j_s}\,\dx, \qquad [M^t]_{k_t, j_t} := \int_{t^n}^{t^{n+1}} \psi^t_{k_t}\psi^t_{j_t} \,\dt, \qquad [E^t]_{k_t} := \psi_{k_t}^t(t^{n}),\\
			&[D^s]_{k_s,j_s}:=  \psi^s_{k_s}(x_{i+\frac12}) \psi^s_{j_s}(x_{i+\frac12}) -\psi^s_{k_s}(x_{i-\frac12}) e^{-\icomplex \theta}\psi^s_{j_s}(x_{i+\frac12}) -\int_{\Omega_i} \partial_x\psi^s_{k_s}(x) \psi^s_{j_s}(x)\,\dx, \\
			& [A^t]_{k_t, j_t} := \psi^t_{k_t}(t^{n+1})\psi^t_{j_t}(t^{n+1})-\int_{t^n}^{t^{n+1}} \partial_t \psi^t_{k_t}(t)\psi^t_{j_t}(t) \,\dt,
		\end{aligned}
	\]
	so that the scheme becomes, with Einstein notation,
	\begin{equation*}
			[A^t]_{k_t,j_t}[M^s]_{k_s,j_s} \hat{q}_{(j_s,j_t)} -[M^s]_{k_s,j_s}[E^t]_{k_t} \hat{u}_{j_s} + [M^t]_{k_t,j_t}[D^s]_{k_s,j_s} \hat{q}_{(j_s,j_t)}=0.
		\end{equation*}
	By inverting the mass matrices, and using the fact that $[A^t]^{-1} E^t $ gives a vector of ones~\cite{veiga2024improving}, we obtain 
	\begin{equation*}
			\hat{q}_{(z_s,z_t)} -\hat{u}_{z_s} + ([A^t]^{-1}_{z_t,k_t} [M^t]_{k_t,j_t})([M^s]^{-1}_{z_s,k_s}[D^s]_{k_s,j_s}) \hat{q}_{(j_s,j_t)}=0.
		\end{equation*}
	This can be written as a Runge--Kutta scheme with time stages $k_t$ for the ODE semidiscrete problem
	\[
	\partial_t \hat{q}_{z_s} + ([M^s]^{-1}_{z_s,k_s}[D^s]_{k_s,j_s}) \hat{q}_{j_s}=0.
	\]
	In particular, it has been shown \cite{offner2025analysis} that the Runge--Kutta method given by the matrix $[A^t]^{-1}_{z_t,k_t} [M^t]_{k_t,j_t}$ is A-stable. \textcolor{black}{So, the method is unconditionally stable if the real part of the eigenvalues of $B = -[M^s]^{-1}\cdot D^s \in \mathbb{C}^{(N+1)\times (N+1)}$ are all nonpositive. We prove this condition by showing that $B$ is negative semidefinite in the scalar product defined by $\left\langle \hat{v}, \hat{w} \right\rangle_{M^s}:= \hat{v}^*M^s\hat{w}$ because for Lagrange basis functions the spatial mass matrix is positive definite (here, $\hat{v}^*$ denotes the conjugate transpose).
	So, for any $\hat{v}\in \mathbb{C}^{N+1}$, upon defining $v(x) := \hat{v}_{j_s} \psi^s_{j_s}(x)$, we have that
	\[
	\begin{aligned}
			\text{Re}\left(\left\langle  \hat{v}, -[M^s]^{-1}[D^s] \hat{v} \right\rangle_{M^s} \right) &=\text{Re}\left( - \hat{v}^* D^s \hat{v} \right) = \text{Re}\left( -\overline{v(x_{i+1/2})} v(x_{i+1/2}) +\overline{v(x_{i-1/2})} e^{-\icomplex \theta}v(x_{i+1/2}) +\int_{\Omega_i} \partial_x\overline{v(x)} \cdot v(x)\,\dx \right) \\
			&= -|v(x_{i+1/2})|^2 + \text{Re}\left( \overline{v(x_{i-1/2})} e^{-\icomplex \theta}v(x_{i+1/2}) \right) + \text{Re}\left( \int_{\Omega_i} \partial_x\overline{v(x)} \cdot v(x)\,\dx \right) \\
			&=-|v(x_{i+1/2})|^2 + \text{Re}\left( \overline{v(x_{i-1/2})} e^{-\icomplex \theta}v(x_{i+1/2}) \right) + \frac{|v(x_{i+1/2})|^2}{2}-\frac{|v(x_{i-1/2})|^2}{2} \\
			&\leq - \frac{|v(x_{i+1/2})|^2}{2}-\frac{|v(x_{i-1/2})|^2}{2}+|v(x_{i-1/2})|\,|v(x_{i+1/2})| 
			=-\frac12 (|v(x_{i-1/2})|-|v(x_{i+1/2})|)^2 \leq 0.
		\end{aligned}
	\]
}
	This shows that all eigenvalues of the semidiscrete operator fall in the left complex half-plane, hence, the method is von Neumann stable independently of the time step.
\end{proof}

\subsection{\textcolor{black}{Numerical convergence  order and accuracy of the implicit method}}
\label{ssec.convergence-classical-implicit}

\begin{table}[t]
	\centering
	\caption{\textcolor{black}{We consider a unit-velocity LAE over $[-6,6]$ with initial condition $e^{-x^2}$ at time $t = 0$. The spatial discretization is fixed at $N_e = 24$ elements, and the implicit method is run for varying degrees $N$ and varying $\CFL$ values, leading to a time step $\Delta t = \Delta x \cdot \CFL = 0.5 \cdot \CFL$. We report the $L^2$-error of the numerical solution at final time $T = 12$, after $N_t = T/\Delta t = 24/\CFL$ time steps are performed. We observe how the method remains stable but progressively becomes more and more dissipative as $\CFL$ increases.}}
	\label{tab:res_implicit}
		\begin{tabular}{l ccccccc}
			\toprule
			{} & {$\CFL = 0.2$} & {$\CFL = 0.5$} & {$\CFL = 1.0$} & {$\CFL = 2.0$} & {$\CFL = 3.0$} & {$\CFL = 6.0$} & {$\CFL = 12.0$} \\
			\midrule
			$N = 1$ & 1.82e-01 & 2.49e-01 & 3.49e-01 & 5.13e-01 & 6.29e-01 & 8.07e-01 & 9.14e-01 \\
			$N = 2$ & 1.05e-02 & 1.98e-02 & 3.68e-02 & 1.76e-01 & 3.13e-01 & 5.81e-01 & 8.40e-01 \\
			$N = 3$ & 4.60e-04 & 1.12e-03 & 3.25e-03 & 3.25e-02 & 9.48e-02 & 3.73e-01 & 6.73e-01 \\
			$N = 4$ & 3.40e-05 & 9.31e-05 & 8.17e-04 & 3.69e-03 & 2.67e-02 & 2.01e-01 & 5.54e-01 \\
			$N = 5$ & 2.83e-06 & 9.89e-06 & 1.50e-04 & 7.66e-04 & 5.28e-03 & 9.93e-02 & 4.28e-01 \\
			$N = 6$ & 2.08e-07 & 1.20e-06 & 9.76e-06 & 2.91e-04 & 9.47e-04 & 4.17e-02 & 3.12e-01 \\
			$N = 7$ & 1.47e-08 & 1.69e-07 & 1.38e-06 & 7.13e-05 & 2.85e-04 & 1.68e-02 & 2.22e-01 \\
			$N = 8$ & 1.05e-09 & 2.35e-08 & 2.38e-07 & 6.86e-06 & 1.05e-04 & 5.68e-03 & 1.50e-01 \\
			$N = 9$ & 6.18e-11 & 1.39e-09 & 2.57e-08 & 7.33e-07 & 2.46e-05 & 2.26e-03 & 9.70e-02 \\
			\bottomrule
		\end{tabular}
\end{table}

\begin{figure}[t]
	\centering
	\begin{subfigure}{0.49\linewidth}
		\includegraphics[width=\linewidth]{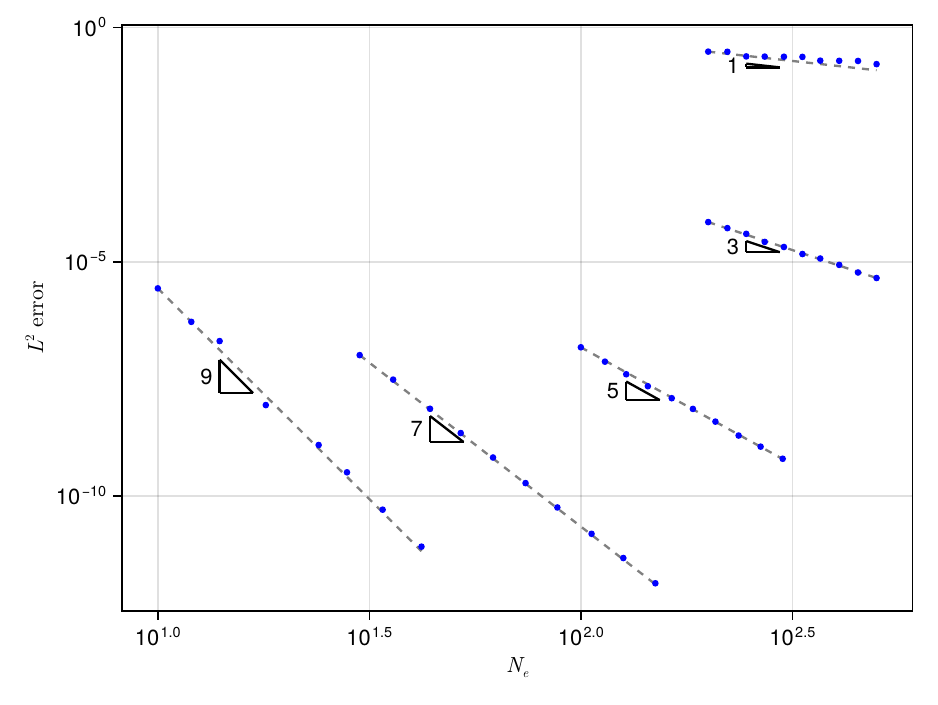}
	\end{subfigure}
	\hfill
	\begin{subfigure}{0.49\linewidth}
		\includegraphics[width=\linewidth]{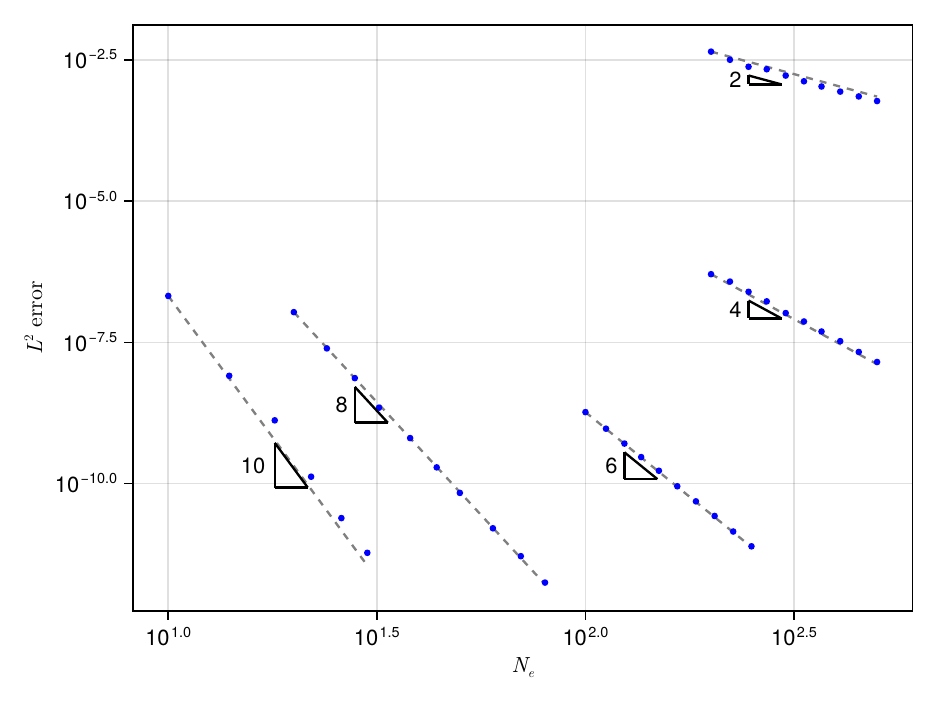}
	\end{subfigure}
	\caption{Convergence  order of the implicit ADER-DG method: left odd orders for $N = 0,2,4,6,8$, right even orders for $N = 1,3,5,7,9$.}
	\label{fig:aderconsistencyimplict}
\end{figure}

The implicit ADER-DG method, as its explicit counterpart, is expected to have order of convergence  $N+1$ for a given polynomial degree $N$. We verify this by means of the same LAE evolution test case described in Section~\ref{ssec.convergence-classical-explicit}. We vary the number of elements $N_e$ and we consider
a $\CFL$ at $10 \cdot \CFL_{\textup{max}}$ where $\CFL_{\textup{max}}$ is taken from the second row of Table~\ref{PNPN}. In Figure~\ref{fig:aderconsistencyimplict}, we report, in logarithmic scale, the $L^2$ norm of the error with respect to the exact solution at time $T=1$.
Each convergence  order is correctly achieved as expected.

\textcolor{black}{While the method is expected to be stable for any $\text{CFL}$ number, arbitrarily large values cannot be chosen if an accurate long-term simulation is desired. This is illustrated in Table~\ref{tab:res_implicit}, where again the same LAE evolution test case described in Section~\ref{ssec.convergence-classical-explicit} is evaluated on a fixed mesh of $N_e = 24$ spatial elements up to $T = 12$ (corresponding to one full domain revolution). As the $\text{CFL}$ increases, the scheme requires fewer time steps and, while remaining stable, progressively loses accuracy due to increased numerical dissipation.}



\section{ALE ADER-DG methods on degenerate spacetime geometries}
\label{sec.degenerate}
In this section, we consider a modification of the ALE ADER-DG scheme on degenerate spacetime geometries, which are obtained by introducing in our computational domain \emph{fictitious spacetime degenerate elements}, which will be termed \emph{sliver} elements.

\medskip
\noindent
\emph{Geometry description.}

\begin{figure}[t]
	\includegraphics[width=0.31\linewidth,height=0.25\textheight]{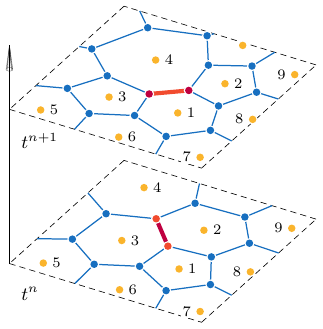}\hfill
	\includegraphics[width=0.31\linewidth,height=0.25\textheight]{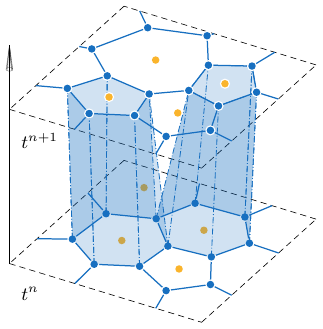} \hfill
	\includegraphics[width=0.31\linewidth,height=0.25\textheight]{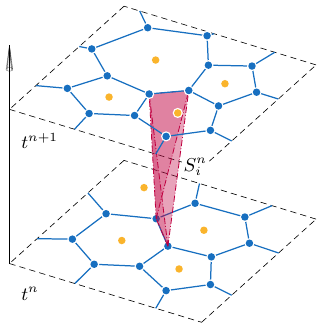}
	\caption{Spacetime connectivity with a topology change in two dimensions (left), induced hole-like sliver element (right) and neighbouring classical control volumes (middle).}
	\label{fig.sliver2d}
\end{figure}

\begin{figure}[t]
	\centering
	\begin{tikzpicture}
		
		\draw[black]   (0.00 , 0.00) --
		(6.00 , 0.00) --
		(5.75 , 2.00) --
		(6.75 , 4.00) --
		(0.25 , 4.00) --
		(0.50 , 2.00) --
		cycle;
		\draw[dotted] (-1.00, 0.00) -- (0.00, 0.00);
		\draw[dotted] ( 6.00, 0.00) -- (7.00, 0.00);
		\draw[dotted] ( 6.75, 4.00) -- (7.75, 4.00);
		\draw[dotted] (-0.75, 4.00) -- (0.25, 4.00);
		
		\draw[dotted] ( 0.00, 0.00) -- (-0.25, 2.00);
		\draw[dotted] (-0.25, 2.00) -- ( 0.25, 4.00);
		\draw[dotted] ( 6.00, 0.00) -- ( 7.00, 2.00);
		\draw[dotted] ( 7.00, 2.00) -- ( 6.75, 4.00);

		\draw[dotted] (4.25,2) -- (2.25,2);
		
		\draw[red] (3,0) -- (4.25,2) -- (3.5,4) -- (2.25,2) -- cycle;
		
		\node at  (1.5,-0.3) {$\Omega_{i-1}^n$};
		\node at  (4.5,-0.3) {$\Omega_{i}^n$};
		
		\node at  (2,4.3) {$\Omega_{i-1}^{n+1}$};
		\node at  (5.12,4.3) {$\Omega_{i}^{n+1}$};
		
		\node at (3,-0.3) {$x_{i-\frac12}^n$};
		\node at (3.5,4.3) {$x_{i-\frac12}^{n+1}$};
		
		\node at (3.4,2.5) {$S_{i-\frac12}^n$};
		\node at (1.5,1) {$C_{i-1}^n$};
		\node at (5,1) {$C_{i}^n$};
		
		\node at (4.33,3) {$\Sigma_{i-\frac12}^{n,+}$};
		\node at (2.3,3) {$\Sigma_{i-\frac12}^{n,-}$};
		\node at (0.88,2.33) {$\Sigma_{i-\frac32}^{n,+}$};
		\node at (6.2,1.6) {$\Sigma_{i+\frac12}^{n,-}$};
		
		\node at (3.2,1.72) {\small $\Delta x_{i-\frac12}^n$};
		
		\node at (7.3,0) {$t^n$};
		\node at (8.25,4) {$t^{n+1}$};
	\end{tikzpicture}
	\caption{Prototype configuration of a spacetime domain discretization with a degenerate (or sliver) element. At each time $t^n$ the spatial domain is subdivided into elements $\Omega_i^n$, connected in spacetime by control volumes $C_i^n$. Interfaces between control volumes are expanded into hole-like sliver elements $S_{i-1/2}^n$ with maximal width $\Delta x_{i-1/2}^n$, mimicking the 3d hole-like sliver elements described in Figure~\ref{fig.sliver2d}.}
	\label{fig.degenerate}
\end{figure}

\noindent
To construct a \emph{degenerate} spacetime geometry we start from the classical geometry described in Figure~\ref{fig.geometry}. As in Figure~\ref{fig.degenerate}, we assume now that some of the interfaces $\Sigma_{i-1/2}^n$ are expanded into a \emph{sliver element} $S_{i-1/2}^n$, which represents our fictitious spacetime degenerate element. To construct a sliver element $S_{i-1/2}^n$, we first fix its maximal width $\Delta x_{i-1/2}^n > 0$ and then consider the quadrilateral whose vertices, in counter-clockwise order, are defined as
\begin{equation*}
	\begin{aligned}
		&(x_{i-1/2}^n, t^n), &( (x_{i-1/2}^n + x_{i-1/2}^{n+1} + \Delta x_{i-1/2}^n)/2, t^n + \Delta t^n/2), \\
		& (x_{i-1/2}^{n+1}, t^{n+1}), &( (x_{i-1/2}^n + x_{i-1/2}^{n+1} - \Delta x_{i-1/2}^n)/2, t^n + \Delta t^n/2).
	\end{aligned}
\end{equation*}
Each neighbouring control volume is redefined accordingly. The boundary $\partial S_{i-1/2}^n$ of each sliver element can be viewed as the union of $\Sigma_{i-1/2}^{n,-}$ (the interface between $S_{i-1/2}^n$ and $C_{i-1}^n$) and $\Sigma_{i-1/2}^{n,+}$ (the interface between $S_{i-1/2}^n$ and $C_{i}^n$). In case an interface $\Sigma_{i-1/2}^n$ is not replaced by a sliver element, then we set $\Sigma_{i-1/2}^{n,\pm} := \Sigma_{i-1/2}^n$.

\begin{remark}
	These degenerate elements mimic in $1$d the hole-like sliver elements introduced in \cite{gaburro2020high, gaburro2021bookchapter, gaburro2021unified, gaburro2025high,Gaburro2027} within the framework of $2$d/$3$d moving meshes with topology changes, see Figure~\ref{fig.sliver2d}. Those works are all based on the use of a direct ALE approach for the evolution of the PDE and therefore on the integration over control volumes that connect in time two different meshes. When, in this context, the meshes do not merely move but can also change topology, meaning that connectivity and list of neighbours of each cell may change, the use of classical control volumes to connect the meshes is no longer sufficient. To cover the entire spacetime domain in between, it then becomes necessary to introduce additional volumes, such as those shown in the right panel of Figure~\ref{fig.sliver2d}, which, similarly to the $1$d degenerate elements of this paper, have zero area at both times $t^n$ and $t^{n+1}$, but a non zero spacetime volume. At present, the $2$d sliver elements have been used in a variety of applications, but a theoretical analysis of their properties has not yet been carried out in detail.
\end{remark}

\medskip
\noindent
\emph{Basis functions.}

\noindent
For each $n > 0$ and $i \in \{1,\dots,N_e\}$, we consider the already introduced families of basis functions $\{\phi_{i,\ell}^n\}_\ell$, $\{\theta_{i,\ell}^n\}_\ell$ and~$\{\psi_{i,\ell}^n\}_\ell$, see \eqref{eq.spacebasis}, \eqref{eq.spacetimebasis} and \eqref{eq.movingbasis}. Furthermore, we also consider over each sliver element $S_{i-1/2}^n$ the family of \emph{spacetime} basis function $\{\theta_{i-1/2,\ell}^n\}_\ell$ defined as
\begin{equation*}
	\begin{aligned}
		&\theta_{i-\frac12,\ell}^n \colon S_{i-\frac12}^n \to \R, \, (x,t) \mapsto \left(\frac{x - x_{i-\frac12}^n}{\Delta x_{i-\frac12}^n} \right)^{\ell_1} \, \left( \frac{t-t^n}{\Delta t^n}\right)^{\ell_2}
		\\[0.6ex]
		&\text{for } \ell = \ell_1 + \ell_2(2N + 3 - \ell_2)/2, \,\, 0 \leq \ell_1 + \ell_2 \leq N.
	\end{aligned}
\end{equation*}
These are nothing else than modal spacetime basis functions centred at $(x_{i-1/2}^n, t^n)$ and scaled according to $\Delta x_{i-1/2}^n$ and $\Delta t^n$. Alongside the general approximants $u_i^n$ and $q_i^n$ described in \eqref{eq.un} and \eqref{eq.qn}, we also consider approximants defined on each sliver element $S_{i-1/2}^n$ to take the form of a spacetime polynomial function $q_{i-1/2}^n$ written as
\begin{equation}\label{eq.qn.sliver}
	q_{i-\frac12}^n(x,t) = \sum_{\ell = 0}^{N_{st}-1} \theta_{i-\frac12,\ell}^n (x, t) \hat{q}_{i-\frac12,\ell}^n \quad \text{for } (x,t) \in S_{i-\frac12}^n, \,i \in \{2,\dots, N_e\}, n > 0,
\end{equation}
where $\{ \hat{q}^{n}_{i-1/2,\ell} \}_{n,i,\ell} \subset \R$ are the degrees of freedom. Again, we write $\hat{q}_{i-1/2}^n$ to identify the corresponding vector of coefficients.

\subsection{Explicit ALE ADER-DG on degenerate geometries}
The explicit ALE ADER-DG scheme on degenerate geometries can be described as a three steps scheme, where the third additional step is devoted to the treatment of the newly introduced sliver elements.

\medskip
\noindent
\emph{Predictor step on control volumes.}

\noindent
The predictor step of the explicit ALE ADER-DG scheme on non-degenerate control volumes remains essentially the same as the one described in Section \ref{sec.predictor}. The only difference is the possible non-quadrilateral concave structure of some control volumes $C_i^n$ that should be accounted for during the numerical integration. Thus, we can easily apply \eqref{eqn.predictor} to obtain each predictor $q_i^n$, $i = 1,\dots, N_e$.

\medskip
\noindent
\emph{Locally implicit step on sliver elements.}

\noindent
Once predictors on non-degenerate control volumes are computed, we consider the sliver elements. Let us fix a sliver element $S_{i-1/2}^n$ for some $i \in \{2,\dots,N_e\}$. On such an element, since we have no inflow information to work with at time $t^n$ (due to the sliver degeneracy), we perform a complete integration by parts that allows to introduce boundary fluxes through flux exchanges at the interfaces $\Sigma_{i-1/2}^{n,\pm}$. Hence, we follow the same derivation that led to \eqref{eq.implicit} and we seek for an approximant $q_{i-1/2}^{n}$ of the form \eqref{eq.qn.sliver} that solves
\begin{equation}\label{eqn.qnsliver}
	\begin{aligned}
		& \int_{\Sigma_{i-\frac12}^{n,-}}
		\theta_{i-\frac{1}{2},k}^n \mathcal{F}(q_{i-\frac{1}{2}}^{n},q_{i-1}^{n}, \mathbf{\hat n}_{i-\frac12}^n)\, \ds
		+ \int_{\Sigma_{i-\frac12}^{n,+}}
		\theta_{i-\frac12,k}^n  \mathcal{F}(q_{i-\frac12}^{n},q_{i}^{n}, \mathbf{\hat n}_{i-\frac12}^n)\, \ds \\
		&= \int_{S_{i-\frac12}^n} 
		\partial_t \theta_{i-\frac12,k}^n(x,t) q_{i-\frac12}^n(x,t) \,\dx \dt + \int_{S_{i-\frac12}^n}  \partial_x \theta_{i-\frac12,k}^n(x,t) f_{i-\frac12}^n(x,t) \,\dx \dt \quad  \text{for }k = 0, \dots, N_{st}-1,
	\end{aligned} 
\end{equation}
where $f_{i-1/2}^n$ is defined as in \eqref{eq.fhat} (by replacing the indices accordingly) and $\mathbf{\hat n}_{i-1/2}^n$ is the outward pointing normal to $\partial S_{i-1/2}^n$. By taking into account that predictors on the neighbouring non-degenerate control volumes have already been computed, we can approximate the only unknown $\hat q_{i-1/2}^{n}$ by means of solving the algebraic system in \eqref{eqn.qnsliver}.

\medskip
\noindent
\emph{Corrector step.}

\noindent
On each element $\Omega_{i}^{n+1}$, $i = 1,\dots,N_e$, the new approximant $u_i^{n+1}$ is then computed by means of \eqref{eqn.corrector}, with the needed substitutions of predictors and interfaces whenever the control volume neighbours a sliver element. We remark that no additional computations are required on sliver elements because they have zero measure boundary at time $t^{n+1}$. Moreover, the method is conservative by construction also around sliver elements thanks to the exchange of fluxes already performed in the computation of their predictors: indeed, the balance of the fluxes can be obtained by taking $\theta_{i-1/2,0}^n = 1$ as test function in \eqref{eqn.qnsliver}.

\subsection{Implicit ALE ADER-DG on degenerate geometries}\label{sec.ALEADERimplicitsliver}
To recover the implicit version of the ALE ADER-DG scheme on degenerate geometries, it is enough to consider together \eqref{eq.implicit} on non-degenerate elements and \eqref{eqn.qnsliver} on sliver elements, by paying attention to use the correct neighbours and interfaces in \eqref{eq.implicit}. Again, each new approximant $u_{i}^{n+1}$ is then recovered by means of an $L^2$ projection of $q_i^n(\cdot,t^{n+1})$ with respect to the basis $\{\phi_{i,\ell}^{n+1}\}_\ell$.

\subsection{Stability analysis with spacetime hole-like sliver elements}
\label{ssec.stability-degenerate}

The von Neumann stability analysis can be performed in the same fashion as outlined in Section~\ref{sec.stability-classical}. Again, we consider an initial datum of the form $e^{\icomplex\kappa x}$.
We build our domain starting from a classical Eulerian geometry and replace every other vertical interface with a sliver-like element, whose diameter is controlled by a parameter $\delta \in [0,0.5]$ (see Figure~\ref{fig:domain_geometry}). Now, we identify two subsequent elements $i-1$ and $i$ separated by a sliver to be our unitary periodic block.

\begin{SCfigure}[2.0][t]
	\begin{tikzpicture}[scale=1.5, baseline=(current bounding box.center)]
		\draw[white] (-4.75,-1) rectangle (-4.75,-1);
		
		\draw[black] (-4,-1) rectangle (0,1);
		
		\draw[black] (-2,1) -- (-1.5,0) -- (-2,-1) -- (-2.5,0) -- cycle;
		
		\node at (-3,-1.15) {\tiny{$\Delta x$}};
		\node at (-1,-1.15) {\tiny{$\Delta x$}};
		\node at (-4.25,0) {\tiny{$\Delta t$}};
		\draw[black, dotted] (-2.5,0) -- (-1.5,0);
		\node at (-2.0,0.1) {\tiny{$2 \delta \Delta x$}};
		
		\draw[->] (-1.75,0.5) -- ++(0.4,0.2);
		\node at (-1,0.75){\tiny{$\left(\frac{\Delta t}{2}, \delta\Delta x\right)$}};
		
		\draw[->] (-1.75,-0.5) -- ++(0.4,-0.2);
		\node at (-1.0,-0.75) {\tiny{$\left(\frac{\Delta t}{2}, -\delta\Delta x\right)$}};
		
		\draw[->] (-2.25,-0.5) -- ++(-0.4,-0.2);
		\node at (-3.05,-0.75) {\tiny{$\left(-\frac{\Delta t}{2}, -\delta\Delta x\right)$}};
		
		\draw[->] (-2.25,0.5) -- ++(-0.4,0.2);
		\node at (-3.0,0.75) {\tiny{$\left(-\frac{\Delta t}{2}, \delta\Delta x\right)$}};
		
		\node at (-0.9,0.0) {\tiny{$C_i^n$}};
		\node at (-3.1,0.0) {\tiny{$C_{i-1}^n$}};
	\end{tikzpicture}
	
	\caption{Periodic building block for the von Neumann analysis: we start from a classical Eulerian geometry and replace every other vertical interface with a sliver-like element, whose diameter is controlled by a parameter $\delta \in [0,0.5]$. Up to scaling, outgoing normal vectors from the sliver-like element take the form reported in the picture.}
	\label{fig:domain_geometry}
\end{SCfigure}

\medskip
\noindent
\emph{Explicit method.}

\noindent
We study the update of the couple $\hat v_i^n = (\hat u_{i-1}^n, \hat u_i^n)$, which encodes the coefficients of the approximate solution on the given periodic block. While the derivation is somehow cumbersome, we can easily prove that the vector $\hat v_i^n \in \R^{2(N+1)}$ follows an update rule of the form
\[
\hat v_i^{n+1} = B_N(\CFL, \delta, \theta) \cdot \hat v_i^n,
\]
with an amplification matrix $B_N(\CFL, \delta, \theta)\in \mathbb C^{2(N+1)\times 2(N+1)}$ dependent only on $\theta = \kappa \Delta x$, on the parameter $\delta$ controlling the width of the sliver and on $\CFL = a\Delta t/\Delta x$. Thus, for a fixed couple $(\CFL, \delta)$, a necessary condition for the update formula for $\hat{v}_i^n$ to be stable, regardless of the value of $\Delta x$, is
\begin{equation}\label{eq.stabilityconditionsliver}
	\rho(B_N(\CFL, \delta, \theta)) \leq 1   \quad \text{for every } \theta \in [0,\pi].
\end{equation}
Following the same approach as in Section \ref{sec.stability-classical}, we introduce, for each $m \in \N$, the discrete amplification factor
\begin{equation}\label{eq.discretestabilityconditionsliver}
	\rho_{N,m}^\delta(c) = \max_{k = 1,\dots,2^{m}+1} \rho\left(B_N\left(c, \delta, \frac{k-1}{2^{m}}\pi\right)\right) \quad \text{for } c \in [0,1], \delta \in [0,0.5].
\end{equation}

We report in Figure~\ref{fig:stabilitysliverexplicitconfronto} the function $\rho_{N,6}^\delta-1$ for varying $\delta = 0, 0.002, 0.004, 0.006, 0.008$. We observe that the discrete amplification factor for $\delta \approx 0$ always lies below the classical amplification factor $\rho_{N,6} = \rho_{N,6}^0$. Hence, the introduction of the sliver element does not increase the amplification factor of the classical geometry, i.e., the same $\CFL$ bounds that we deem acceptable in the classical setting can be extended to the degenerate setting. We report the full analysis in Figure~\ref{fig:stabilitysliverexplicit}, where we identify the boundary at which $\rho_{N,6}^\delta$ crosses the level $1+10^{-3}$ (according to the strategy employed to determine the second row in Table~\ref{PNPN}). We observe how the introduction of the hole-like sliver elements is slightly increasing the maximal acceptable $\CFL$ value: this result is probably a consequence of the implicit treatment of the hole-like element. Indeed, the predictor step on the hole-like sliver element is purely implicit, introducing a locally implicit step in our globally explicit scheme.

\begin{figure}[t]
	\centering
	\includegraphics[width=0.9\linewidth]{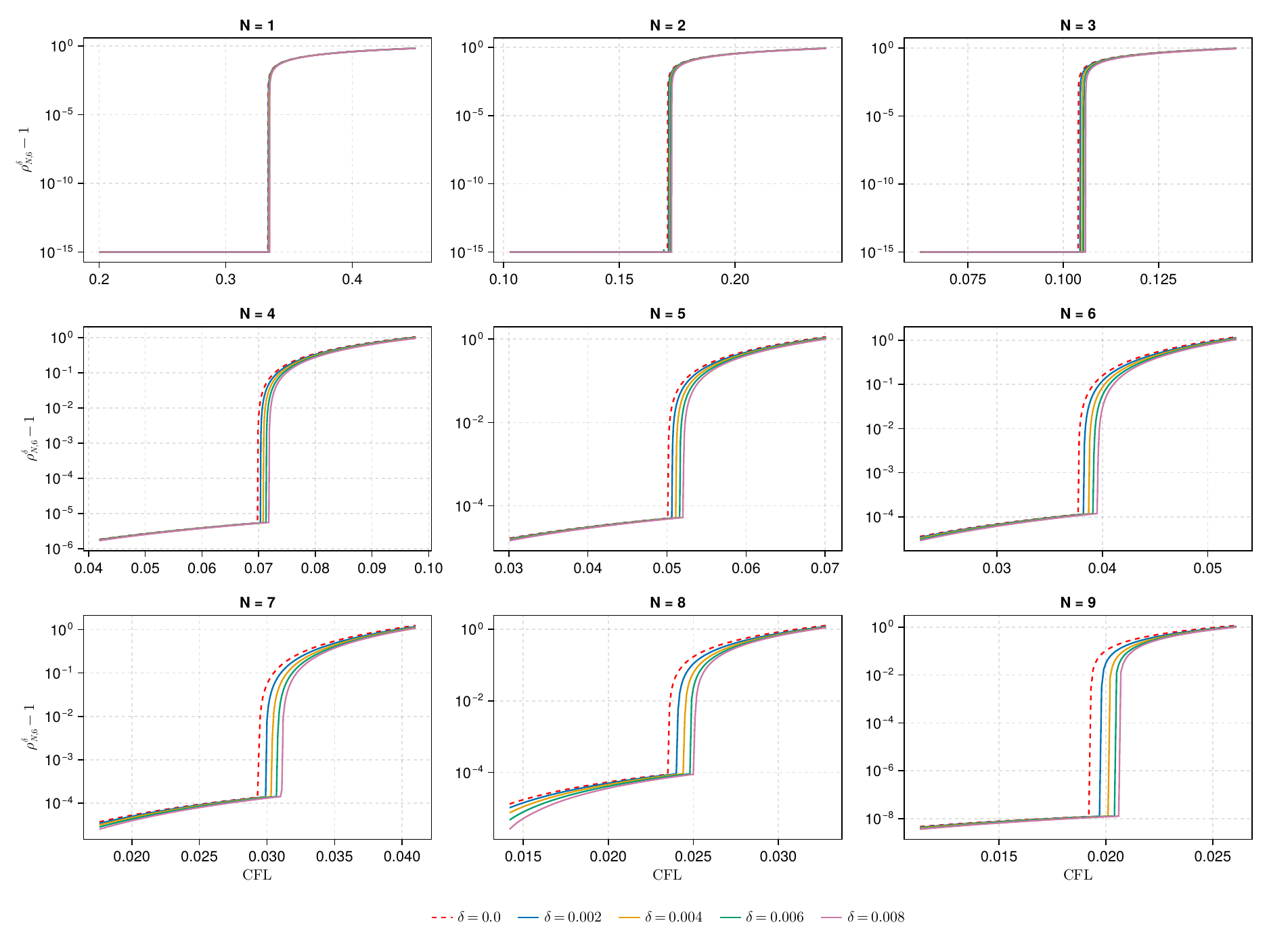}
	\caption{Study of the amplification factor for the explicit ADER-DG method on degenerate geometries. In each plot, we display the function $\CFL \mapsto \rho_{N,6}^0(\CFL) - 1 = \rho_{N,6}(\CFL) - 1$ and the functions $\CFL \mapsto \rho_{N,6}^\delta(\CFL) - 1$ for $\delta = 0.002, 0.004, 0.006, 0.008$. For each $\CFL$ the amplification factor in the degenerate setting is always smaller than the one in the classical setting.}
	\label{fig:stabilitysliverexplicitconfronto}
\end{figure}

\begin{figure}[t]
	\centering
	\includegraphics[width=0.9\linewidth]{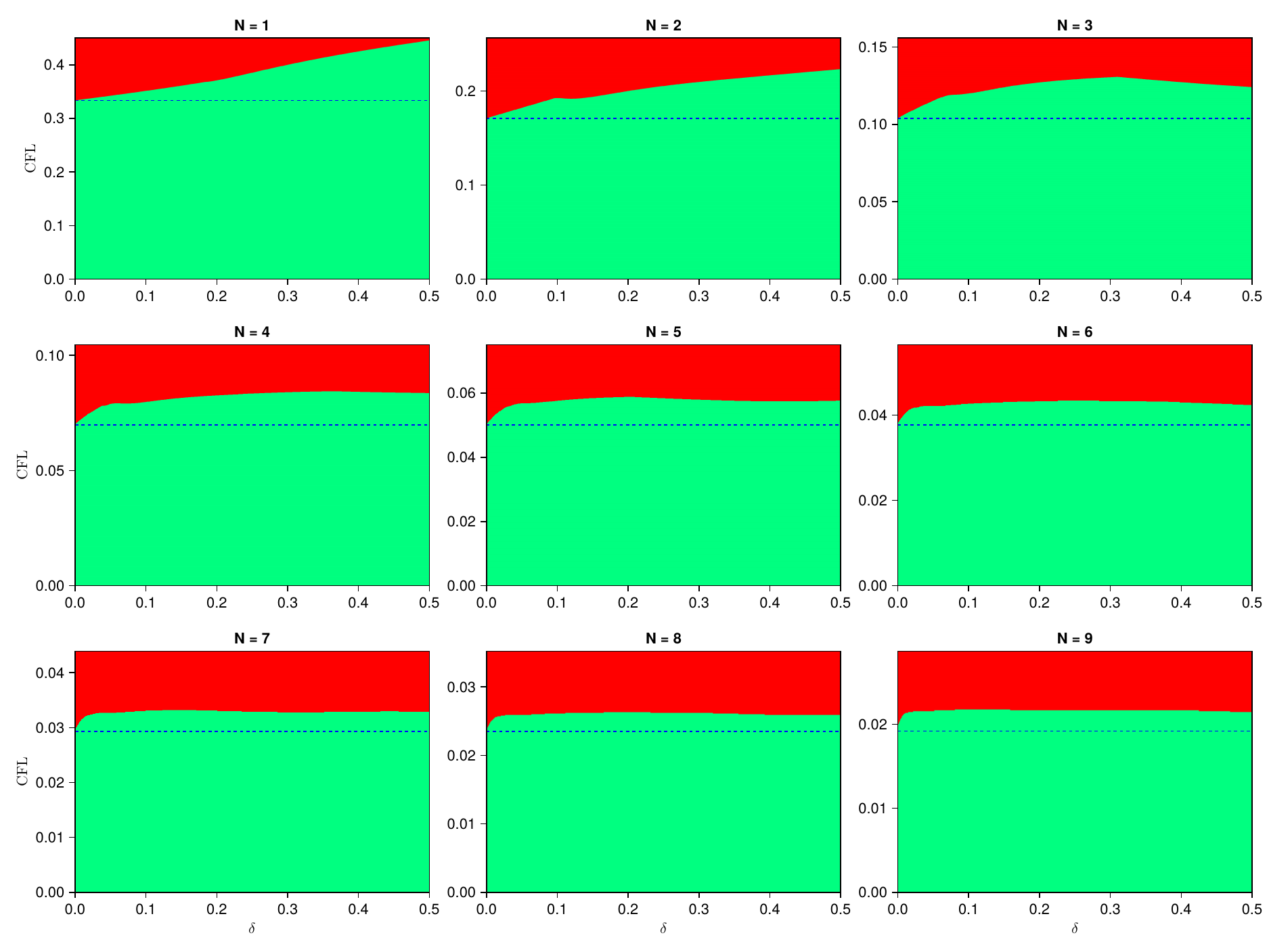}
	\caption{Stability study of the explicit ADER-DG method with slivers. The blue dotted lines are set at the $\CFL_{\rho \leq 1+1\es{-3}}$ values reported in the second row of Table~\ref{PNPN}. Then, in each plot, a point $(\delta, \CFL)$ is either green or red depending on whether $\rho_{N,6}^\delta(\CFL) \leq 1 + 10^{-3}$ or not. We deduce that classical $\CFL$ numbers can be used even in presence of sliver elements.}
	\label{fig:stabilitysliverexplicit}
\end{figure}

\medskip
\noindent
\emph{Implicit method.}

\noindent 
The same von Neumann stability analysis can be developed for the implicit ADER-DG method, by focusing again on the update formula for the couple $\hat v_i^n = (\hat u_{i-1}^n, \hat u_i^n)$ and obtaining the corresponding amplification matrix $B_N(\CFL, \delta, \theta)$.
We evaluate condition \eqref{eq.stabilityconditionsliver} by means of computing $\rho_{N,6}^\delta$, see \eqref{eq.discretestabilityconditionsliver}, on the discrete sets
$
C = \{c = k \cdot 10^{-2} \mid k \in \N \} \cap [0, 10] \text{ and } \Delta = \{\delta = k \cdot 10^{-2} \mid k \in \N \} \cap [0,0.5].
$
We obtain the following:
\begin{center}
	\resizebox{\textwidth}{!}{%
		\begin{tabular}{c|ccccccccc}
			& $N=1$ & $N=2$ & $N=3$ & $N=4$ & $N=5$ & $N=6$ & $N=7$ & $N=8$ & $N=9$ \\ \hline \\[-10pt]
			$\max\limits_{c \in C, \delta \in \Delta} \rho_{N,6}^\delta(c) - 1$ & 3.99\es{-15} & 5.55\es{-15} & 7.55\es{-15} & 2.15\es{-14} & 1.27\es{-13} & 2.01\es{-12} & 1.35\es{-10} & 4.44\es{-9}& 1.21\es{-7} \\ 
		\end{tabular}.%
	}
\end{center}
As is the case in the classical setting, no clear $\CFL$-type bounds can be identified. The largest discrepancy for increasing $N$ can be related to the growing dimension of the matrices involved in the computations, which scale as $3 N_{st} \times 3 N_{st}$ and reach $165 \times 165$ for $N = 9$. These results suggest the unconditional stability of the method. Indeed, following the same idea outlined in the stability proofs in \cite{dumbser2016space, yan2002local1, yan2002local2}, we can actually prove that the method is $L^2$ stable in several configurations of interest.

\begin{theorem}
	Consider the linear advection equation
	\[
	\partial_t u(x,t) + \partial_x u(x,t) = 0 \quad \text{for } x\in(0,2\pi), t>0,
	\]
	coupled with periodic boundary conditions and \textcolor{black}{a periodic real valued initial datum $u(x,0) = u_0(x)$}. Consider over $[0,2\pi]$ a uniform Eulerian grid of an even number $N_e > 0$ of elements of size $\Delta x > 0$. Replace every other interface with a hole-like sliver element of size $2\delta \Delta x$, $\delta \in (0,1)$ (see Figure~\ref{fig:domain_geometry}). Consider the Rusanov-type numerical flux \eqref{eq.rusanov}. Then, when applied to this problem, the implicit ADER-DG method with slivers as described in Section \ref{sec.ALEADERimplicitsliver} is $L^2$ stable provided $$\delta \leq \frac12 \min\left\{ \frac{\Delta t}{\Delta x}, 1\right\}.$$
\end{theorem}

\begin{proof}
	
	We recall the construction in Figure~\ref{fig:domain_geometry}.
	We consider again a periodic degenerate geometry where every other interface is replaced by a sliver element. We focus on a triplet of elements $(C_{i-1}^n, S_{i-1/2}^n, C_i^n)$, with $\delta \in (0,0.5)$.
	Up to renormalization, the normals to the four inner sides of the sliver have the coordinates reported in Figure~\ref{fig:domain_geometry}.
	
	First, we observe that under the assumption $f(q) = q$, the Rusanov-type ALE flux \eqref{eq.rusanov} can be simplified into
	\[
	\mathcal{F}(q^-, q^+, \mathbf{\hat n}) =  
	\frac{1}{2} \left[ ( q^{+} + q^{-} )(\mathbf{\hat n}_x + \mathbf{\hat n}_t) - (q^{+} - q^{-})|\mathbf{\hat n}_x + \mathbf{\hat n}_t| \right],
	\]
	meaning that either $q^{+}$ or $q^{-}$ is selected depending on the sign of $(\mathbf{\hat n}_x + \mathbf{\hat n}_t)$. In particular, given that $\delta \leq \frac12 \frac{\Delta t}{\Delta x}$, in the subsequent steps the numerical flux will always select $q_{i-1}$ when evaluated along $\Sigma_{i-1/2}^{n,-}$ and $q_{i-1/2}^n$ when evaluated along $\Sigma_{i-1/2}^{n,+}$, resulting in a pure upwind scheme from the left to the right element.
	
	For given continuous spacetime functions $w,v$ and for all meaningful indices $i,n$, we define for convenience the following pairings:
	\[
	\begin{aligned}
		&\begin{aligned}
			[w, v]_i^n &:= \int_{\Omega_i^n} w(x, t^n) v(x, t^n) \, \dx, \qquad& \{w, v\}_{i-1/2}^{n} &:= \int_{t^n}^{t^{n+1}} w(x_{i-1/2}, t) v(x_{i-1/2}, t) \, \dt, \\
			\langle w, v \rangle_i^n &:= \int_{C_i^n} w(x,t) v(x,t) \, \dx \dt, \qquad& \langle w, v \rangle_{i-1/2}^n &:= \int_{S_{i-1/2}^n} w(x,t) v(x,t) \, \dx \dt,
		\end{aligned}\\
		&\{w, v\}_{i-1/2}^{n,\pm} := \int_{\Sigma_{i-1/2}^{n,\pm}} w v (\mathbf{n}_t + \mathbf{n}_x) \, \ds \quad \text{where } \mathbf{n} \text{ is the outer normal to $\partial S_{i-1/2}^n$}.
	\end{aligned}
	\]
	Consider now \eqref{eq.implicit} on the element $C_i^n$: by recalling that we obtained each initial status as $u_i^n(x) = q_i^{n-1}(x,t^n)$ for $x \in \Omega_i^n$, plugging in as test function $q_i^n$ itself into \eqref{eq.implicit}, we obtain
	\begin{equation*}
		[q_i^n, q_i^n]_i^{n+1} - [q_i^n, q_i^{n-1}]_i^n - \langle \partial_t q_i^n + \partial_x q_i^n, q_i^n \rangle_i^n - \{q_i^n, q_{i-1/2}^n\}_{i-1/2}^{n,+} + \{q_i^n, q_i^n\}_{i+1/2}^{n} = 0.
	\end{equation*}
	By using that $\partial_t ((q_i^n)^2) = 2q_i^n\partial_t q_i^n$ and $\partial_x ((q_i^n)^2) = 2q_i^n\partial_x q_i^n$, and applying the divergence theorem, we get
	\begin{equation*}
		\begin{split}
			&[q_i^n, q_i^n]_i^{n+1} - [q_i^n, q_i^{n-1}]_i^n - \frac{1}{2}[q_i^n, q_i^n]_i^{n+1} + \frac{1}{2}[q_i^n, q_i^n]_i^n + \frac{1}{2}\{q_i^n, q_i^n\}_{i-1/2}^{n,+} - \frac{1}{2}\{q_i^n, q_i^n\}_{i+1/2}^{n} \\[2pt]
			&- \{q_i^n, q_{i-1/2}^n\}_{i-1/2}^{n,+} + \{q_i^n, q_i^n\}_{i+1/2}^{n} = 0.
		\end{split}
	\end{equation*}
	By multiplying by $2$, adding and subtracting $[q_i^{n-1}, q_i^{n-1}]_i^n$ and $\{q_{i-1/2}^n,q_{i-1/2}^n\}_{i-1/2}^{n,+}$, and rearranging the terms, we obtain
	\begin{equation*}
		\begin{aligned}
			&[q_i^n, q_i^n]_i^{n+1} - [q_i^{n-1}, q_i^{n-1}]_i^n + [q_i^n - q_i^{n-1}, q_i^n - q_i^{n-1}]_i^n \\[2pt]
			&+ \{q_i^n, q_i^n\}_{i+1/2}^{n} + \{1, (q_i^n - q_{i-1/2}^n)^2\}_{i-1/2}^{n,+} - \{q_{i-1/2}^n, q_{i-1/2}^n\}_{i-1/2}^{n,+} = 0.
		\end{aligned}
	\end{equation*}
	Given our choice of $\delta$, we can easily see that $\{1, (q_i^n - q_{i-1/2}^n)^2\}_{i-1/2}^{n,+} \geq 0$, hence, from the equality above, we get the inequality
	\begin{equation}\label{eq.ineq1}
		\begin{aligned}
			[q_i^n, q_i^n]_i^{n+1} - [q_i^{n-1}, q_i^{n-1}]_i^n + \{q_i^n, q_i^n\}_{i+1/2}^{n} - \{q_{i-1/2}^n, q_{i-1/2}^n\}_{i-1/2}^{n,+} \leq 0.
		\end{aligned}
	\end{equation}
	
	Consider now the sliver element $S_{i-1/2}^n$ and the corresponding relation \eqref{eqn.qnsliver}: by using as test function $q_{i-1/2}^n$ itself, we readily obtain
	\begin{equation*}
		-\langle \partial_t q_{i-1/2}^n + \partial_x q_{i-1/2}^n, q_{i-1/2}^n \rangle_{i-1/2}^n + \{q_{i-1/2}^n, q_{i-1/2}^n\}_{i-1/2}^{n,+} + \{q_{i-1/2}^n, q_{i-1}^n\}_{i-1/2}^{n,-} = 0.
	\end{equation*}
	By applying again the divergence theorem to the first term, multiplying by $2$, adding and subtracting $\{q_{i-1}^n, q_{i-1}^n\}_{i-1/2}^{n,-}$ and rearranging the terms, we derive 
	\begin{equation*}
		\{q_{i-1/2}^n, q_{i-1/2}^n\}_{i-1/2}^{n,+} - \{1, (q_{i-1/2}^n - q_{i-1}^n)^2\}_{i-1/2}^{n,-} + \{q_{i-1}^n, q_{i-1}^n\}_{i-1/2}^{n,-} = 0.
	\end{equation*}
	By our choice of $\delta$, we have $\{1, (q_{i-1/2}^n - q_{i-1}^n)^2\}_{i-1/2}^{n,-} \leq 0$, leading to
	\begin{equation}\label{eq.ineq2}
		\begin{aligned}
			\{q_{i-1/2}^n, q_{i-1/2}^n\}_{i-1/2}^{n,+} + \{q_{i-1}^n, q_{i-1}^n\}_{i-1/2}^{n,-} \leq 0.
		\end{aligned}
	\end{equation}
	On the left element $C_{i-1}^n$, the same derivation done on $C_i^n$ applies, and we obtain
	\begin{equation}\label{eq.ineq3}
		\begin{aligned}
			[q_{i-1}^n, q_{i-1}^n]_{i-1}^{n+1} - [q_{i-1}^{n-1}, q_{i-1}^{n-1}]_{i-1}^n - \{q_{i-1}^n, q_{i-1}^n\}_{i-1/2}^{n,-} -  \{q_{i-2}^n, q_{i-2}^n\}_{i-3/2}^{n} \leq 0.
		\end{aligned}
	\end{equation}
	By taking the sum of \eqref{eq.ineq1}, \eqref{eq.ineq2} and \eqref{eq.ineq3}, we get
	\begin{equation*}
		\begin{split}
			&[q_i^n, q_i^n]_i^{n+1} - [q_i^{n-1}, q_i^{n-1}]_{i}^n + [q_{i-1}^n, q_{i-1}^n]_{i-1}^{n+1} - [q_{i-1}^{n-1}, q_{i-1}^{n-1}]_{i-1}^n + \{q_i^n, q_i^n\}_{i+1/2}^{n} - \{q_{i-2}^n, q_{i-2}^n\}_{i-3/2}^{n} \le 0.
		\end{split}
	\end{equation*}
	We sum now over all positive indices $i$ and over time, \textcolor{black}{take into account the periodicity assumption to cancel out boundary terms}, and obtain
	\begin{equation*}
		\int_{\Omega} (q^n(x, t^{n+1}))^2 - (u^0(x))^2 \, \dx \le 0,
	\end{equation*}
	which concludes the proof.
\end{proof}

\subsection{Numerical convergence  with spacetime hole-like sliver elements}
\label{ssec.convergence-degenerate}

\begin{figure}[t]
	\centering
	\begin{subfigure}{0.49\linewidth}
		\includegraphics[width=\linewidth]{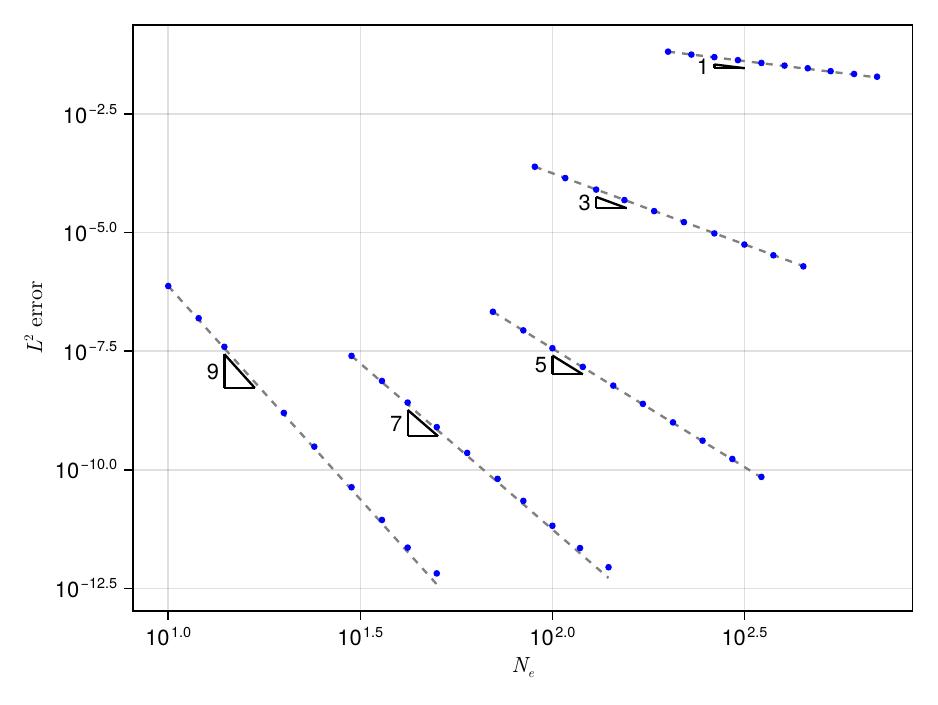}
	\end{subfigure}
	\hfill
	\begin{subfigure}{0.49\linewidth}
		\includegraphics[width=\linewidth]{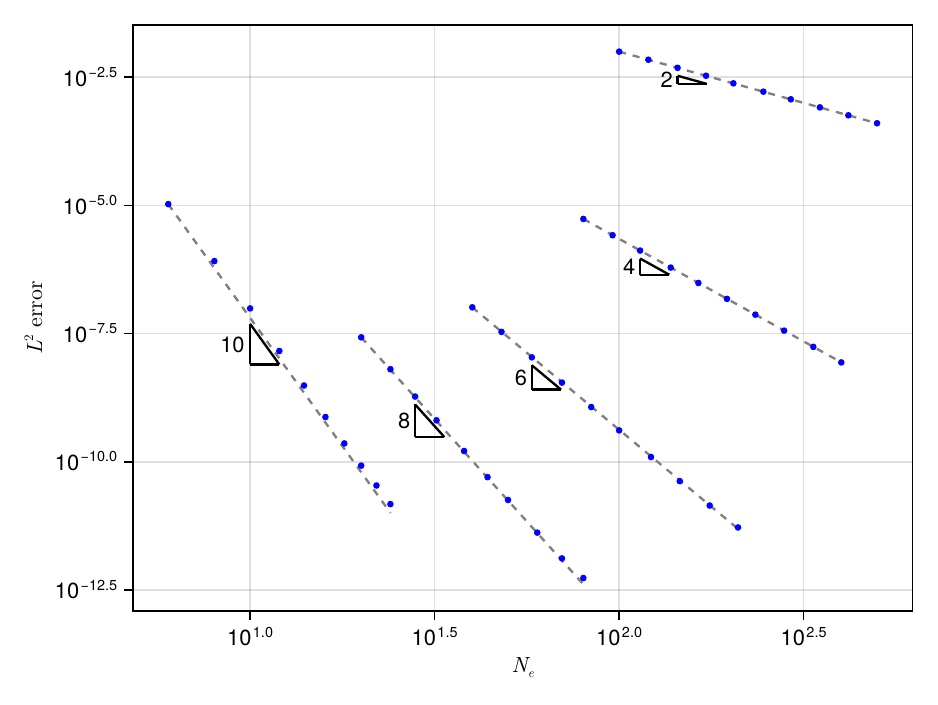}
	\end{subfigure}
	
	\begin{subfigure}{0.49\linewidth}
		\includegraphics[width=\linewidth]{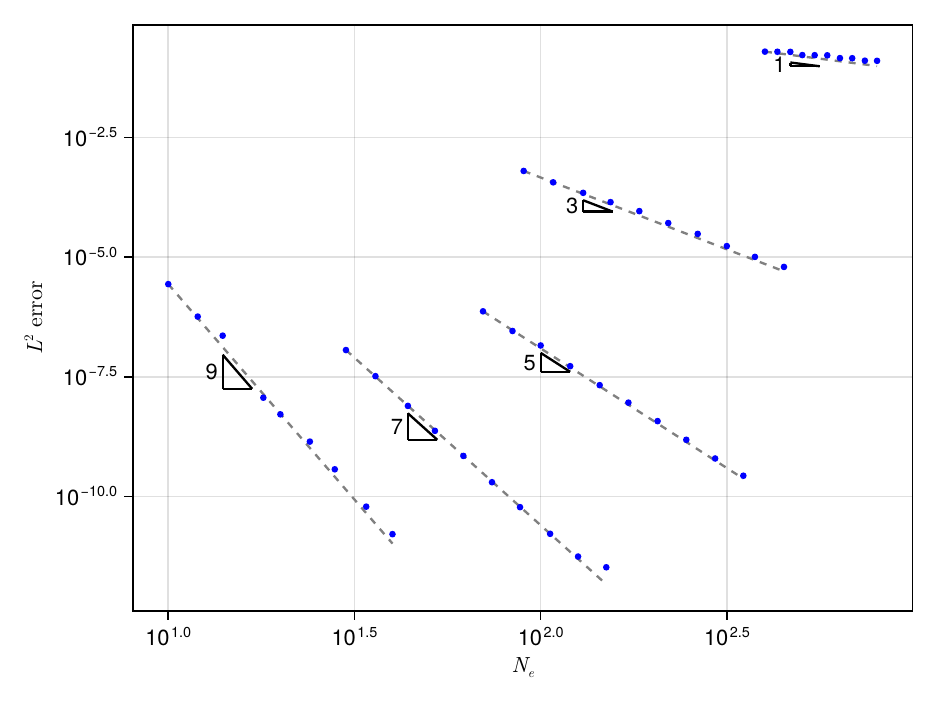}
	\end{subfigure}
	\hfill
	\begin{subfigure}{0.49\linewidth}
		\includegraphics[width=\linewidth]{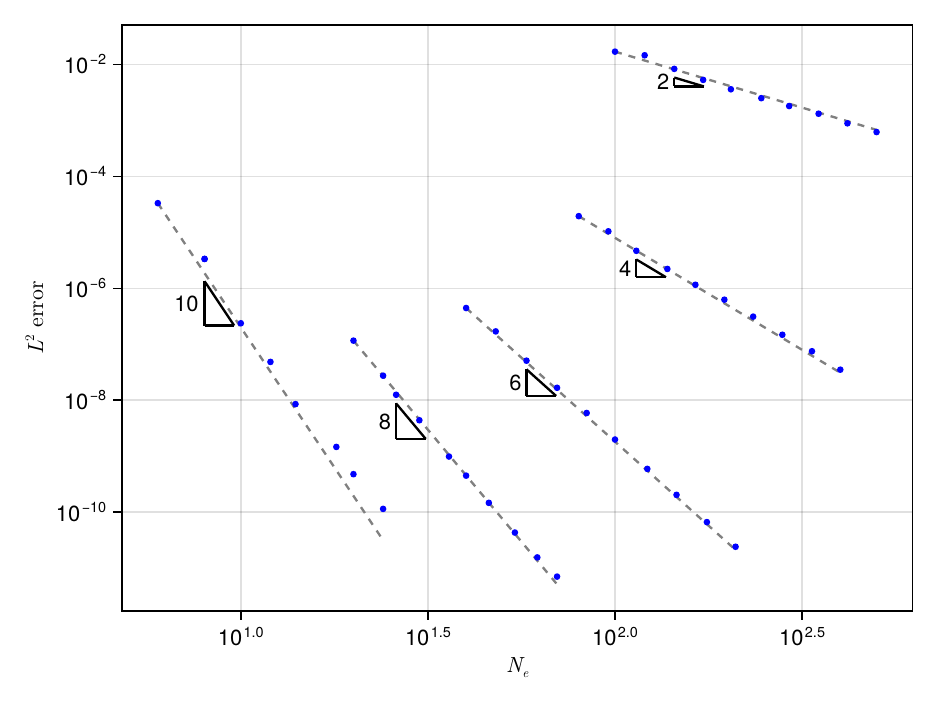}
	\end{subfigure}
	
	\caption{Convergence  order of the explicit (top) and the implicit (bottom) ADER-DG method with slivers: left odd orders for $N = 0,2,4,6,8$, right even orders for $N = 1,3,5,7,9$.}
	\label{fig:sliverconsistency}
\end{figure}

As both the explicit and the implicit ADER-DG methods on degenerate geometries are derived as generalisations of the ones on classical geometries, we expect them to have order of convergence  $N+1$ for a given polynomial degree $N$. To verify this property we repeat the same numerical analysis of Sections~\ref{ssec.convergence-classical-explicit} and \ref{ssec.convergence-classical-implicit}, over a domain where we insert a sliver element at each interface (except at the boundary), with $\delta = 0.2$. As in the classical setting, we set a $\CFL$ at $0.9 \cdot \CFL_{\textup{max}}$ for the explicit method and at $10 \cdot \CFL_{\textup{max}}$ for the implicit method, where $\CFL_{\textup{max}}$ is taken from the second row of Table~\ref{PNPN}.
In Figure~\ref{fig:sliverconsistency}, we report $L^2$ errors with respect to the exact solution at the respective final times ($T = 12$ for the explicit and $T = 1$ for the implicit). Each convergence  order is correctly achieved as expected.

\section{Conclusions and outlook to future works}
\label{sec.conclusions}

In this paper, we studied the von Neumann stability and the convergence  of the family of ADER-DG methods for $N = 1,\dots,9$, considering both explicit and implicit formulations, all within the ALE framework. In particular, we showed that the use of degenerate spacetime geometries leads to some equivalent $\CFL$-type stability bounds as those governing stability in the case of classical geometries. Therefore, elements with zero spatial size do not lead to a reduction of the time step. This result is also important because it provides a theoretical foundation for the use of degenerate elements to connect moving meshes in multiple dimensions, and it represents a natural starting point for the construction of new spacetime cut cell-based methods \cite{may2017explicit, may2024accuracy}.


\section*{Acknowledgments}

M.~Bonafini is member of the INdAM GNAMPA group in Italy; 
D.~Torlo and E.~Gaburro are members of the INdAM GNCS group in Italy. 
E.~Gaburro and M.~Bonafini gratefully acknowledge the support received from the European Union 
with the ERC Starting Grant \emph{ALcHyMiA} (grant agreement No. 101114995).
Views and opinions expressed are however those of the authors only and do not necessarily 
reflect those of the European Union or the European Research Council Executive Agency. 
Neither the European Union nor the granting authority can be held responsible for them.
D.~Torlo was supported by the Ateneo Sapienza project 2024 “Advanced Computational Methods for Real-World Applications: Data-Driven Models, Hyperbolic Equations and Optimal Control”.


\bibliographystyle{abbrv}
\bibliography{literature}

@book{leveque2007finite,
	title={Finite difference methods for ordinary and partial differential equations: steady-state and time-dependent problems},
	author={LeVeque, Randall J},
	year={2007},
	publisher={SIAM}
}

@article{dumbser2008unified,
  title={{A unified framework for the construction of one-step finite volume and discontinuous Galerkin schemes on unstructured meshes}},
  author={Dumbser, M. and Balsara, D.S. and Toro, E.F. and Munz, C.-D.},
  journal={Journal of Computational Physics},
  volume={227},
  number={18},
  pages={8209--8253},
  year={2008},
  publisher={Elsevier}
}

@article{gaburro2020high,
  title={High order direct {Arbitrary-Lagrangian-Eulerian schemes on moving Voronoi meshes with topology changes}},
  author={Gaburro, E. and Boscheri, W. and Chiocchetti, S. and Klingenberg, C. and Springel, V. and Dumbser, M.},
  journal={Journal of Computational Physics},
  volume={407},
  pages={109167},
  year={2020},
  publisher={Elsevier}
}

@article{gaburro2021unified,
  title={{A unified framework for the solution of hyperbolic PDE systems using high order direct Arbitrary-Lagrangian--Eulerian schemes on moving unstructured meshes with topology change}},
  author={Gaburro, E.},
  journal={Archives of Computational Methods in Engineering},
  volume={28},
  number={3},
  pages={1249--1321},
  year={2021},
  publisher={Springer}
}

@article{fambri2020discontinuous,
  title={{Discontinuous Galerkin methods for compressible and incompressible flows on space--time adaptive meshes: toward a novel family of efficient numerical methods for fluid dynamics}},
  author={Fambri, F.},
  journal={Archives of Computational Methods in Engineering},
  volume={27},
  number={1},
  pages={199--283},
  year={2020},
  publisher={Springer}
}

@article{chiocchetti2021high,
  title={{High order ADER schemes and GLM curl cleaning for a first order hyperbolic formulation of compressible flow with surface tension}},
  author={Chiocchetti, S. and Peshkov, I. and Gavrilyuk, S. and Dumbser, M.},
  journal={Journal of Computational Physics},
  volume={426},
  pages={109898},
  year={2021},
  publisher={Elsevier}
}

@article{ReALE2015,
  author = {W. Bo and M.J. Shashkov},
  title = {{Adaptive reconnection-based arbitrary Lagrangian Eulerian method}},
  journal = {Journal of Computational Physics},
  year = {2015},
  volume = {299},
  pages = {902--939}
}

@article{ReALE2010,
  author = {R. Loub{\`e}re and P. H. Maire and M.J. Shashkov and J. Breil and S. Galera},
  title = {{ReALE: A reconnection-based arbitrary-Lagrangian–Eulerian method}},
  journal = {Journal of Computational Physics},
  year = {2010},
  volume = {229},
  pages = {4724--4761}
}

@article{kemm2020simple,
  title={{A simple diffuse interface approach for compressible flows around moving solids of arbitrary shape based on a reduced Baer--Nunziato model}},
  author={Kemm, F. and Gaburro, E. and Thein, F. and Dumbser, M.},
  journal={Computers \& fluids},
  volume={204},
  pages={104536},
  year={2020},
  publisher={Elsevier}
}

@article{tavelli2020space,
  title={{Space-time adaptive ADER discontinuous Galerkin schemes for nonlinear hyperelasticity with material failure}},
  author={Tavelli, M. and Chiocchetti, S. and Romenski, E. and Gabriel, A.-A. and Dumbser, M.},
  journal={Journal of computational physics},
  volume={422},
  pages={109758},
  year={2020},
  publisher={Elsevier}
}

@article{Rusanov:1961a,
	author  = "Rusanov, V. V.",
	title   = "{Calculation of Interaction of Non--Steady
                    Shock Waves with Obstacles}",
	journal = "J. Comput. Math. Phys. USSR",
	year    = "1961",
	volume  = "1",
	pages   = "267--279",
}

@article{boscheri2013arbitrary,
  title={{Arbitrary-Lagrangian-Eulerian one-step WENO finite volume schemes on unstructured triangular meshes}},
  author={Boscheri, W. and Dumbser, M.},
  journal={Communications in Computational Physics},
  volume={14},
  number={5},
  pages={1174--1206},
  year={2013},
  publisher={Cambridge University Press}
}

@inproceedings{gaburro2021bookchapter,
  title={{High-order Arbitrary-Lagrangian-Eulerian schemes on crazy moving Voronoi meshes}},
  author={Gaburro, E. and Chiocchetti, S.},
  booktitle={Young Researchers Conference},
  pages={99--119},
  year={2021},
  organization={Springer}
}

@Article{Castro2006,
author ={M.J. Castro and J.M. Gallardo and C. Par\'es},
title = {High-order finite volume schemes based on reconstruction of states for solving hyperbolic systems with nonconservative products. {A}pplications
to shallow-water systems}, 
journal = {Mathematics of Computation},
volume = 75,
pages = {1103--1134},
year = 2006,
}

@Article{Pares2006,
author ={C. Par\'es},
title = {Numerical methods for nonconservative hyperbolic systems: a theoretical framework}, 
journal = {SIAM Journal on Numerical Analysis},
volume = 44,
pages = {300-321},
year = 2006,
}

@Article{Castro2008,
author ={M.J. Castro and J.M. Gallardo  and J.A. L\'opez and  C. Par\'es},
title = {Well-balanced high order extensions of Godunov's method for semilinear balance laws}, 
journal = {SIAM Journal of  Numerical Analysis},
volume ={46} ,
pages = {1012--1039},
year = 2008,
}

@article{busto2020high,
  title={{High order ADER schemes for continuum mechanics}},
  author={Busto, S. and Chiocchetti, S. and Dumbser, M. and Gaburro, E. and Peshkov, I.},
  journal={Frontiers in Physics},
  volume={8},
  pages={32},
  year={2020},
  publisher={Frontiers Media SA}
}

@article{hidalgo2011ader,
  title={{ADER schemes for nonlinear systems of stiff advection--diffusion--reaction equations}},
  author={Hidalgo, A. and Dumbser, M.},
  journal={Journal of Scientific Computing},
  volume={48},
  number={1-3},
  pages={173--189},
  year={2011},
  publisher={Springer}
}

@article{popov2023space,
  title={{Space-Time Adaptive ADER-DG Finite Element Method with LST-DG Predictor and a posteriori Sub-cell WENO Finite-Volume Limiting for Simulation of Non-stationary Compressible Multicomponent Reactive Flows}},
  author={Popov, I. S.},
  journal={Journal of Scientific Computing},
  volume={95},
  number={2},
  pages={44},
  year={2023},
  publisher={Springer}
}

@article{dumbser2018efficient,
  title={Efficient implementation of ADER discontinuous Galerkin schemes for a scalable hyperbolic PDE engine},
  author={Dumbser, Michael and Fambri, Francesco and Tavelli, Maurizio and Bader, Michael and Weinzierl, Tobias},
  journal={axioms},
  volume={7},
  number={3},
  pages={63},
  year={2018},
  publisher={MDPI}
}

@Article{ADERGRMHD,
author ={F. Fambri and M. Dumbser and S. K\"oppel and L. Rezzolla and O. Zanotti}, 
title = {{{ADER} discontinuous {G}alerkin schemes for general-relativistic ideal magnetohydrodynamics}}, 
journal = {Monthly Notices of the Royal Astronomical Society},
volume = {477},
pages = {4543--4564}, 
year = {2018},
}

@PhdThesis{mill, 
  AUTHOR =       {{R.C.} Millington and {E.F.} Toro and {L.A.M.} Nejad}, 
  TITLE =        {Arbitrary High Order Methods for Conservation Laws 
                  I: The One Dimensional Scalar Case}, 
  SCHOOL =  {Manchester Metropolitan University, Department of 
                  Computing and Mathematics}, 
  YEAR =         {1999}, 
  MONTH =        {June} 
}

@InCollection{toro1,
  author = 	 {{E.F.} Toro and {R.C.} Millington and {L.A.M} Nejad},
  title = 	 {Towards very high order {Godunov} schemes},
  booktitle = 	 {Godunov Methods. Theory and Applications},
  pages =	 {905--938},
  publisher =	 {Kluwer/Plenum Academic Publishers},
  year =	 2001,
  editor =	 {{E.F.} Toro}
}

@Article{toro3,
  author = 	 {{V.A.} Titarev and {E.F.} Toro},
  title = 	 {{ADER}: Arbitrary High Order {Godunov} Approach},
  journal = 	 {Journal of Scientific Computing},
  year = 	 2002,
  volume =	 17,
  number =	 {1-4},
  pages =	 {609--618},
  month =	 {December}
}

@Article{toro4,
  author = 	 {{E.F.} Toro and {V. A.} Titarev},
  title = 	 {Solution of the generalized {Riemann} problem for advection-reaction equations},
  journal = 	 {Proc. Roy. Soc. London},
  volume = {458}, 
  pages = {271-281}, 
  year = {2002}, 
  }

@article{Kaser05,
author = {K{\"a}ser, Martin and Iske, A.},
year = {2005},
month = {01},
pages = {489-508},
title = {Adaptive {ADER} schemes for the solution of scalar non-linear hyperbolic problems},
volume = {205},
journal = {J Comput Phys}
}

@phdthesis{Kaser2003,
  title={Adaptive methods for the numerical simulation of transport processes},
  author={K{\"a}ser, Martin Andreas},
  year={2003},
  school={Technische Universit{\"a}t M{\"u}nchen}
}

@Article{DumbserKaeser07,
author  ={M. Dumbser and M. K\"aser and V.A Titarev and E.F. Toro},
title   = {Quadrature-Free Non-Oscillatory Finite Volume Schemes on Unstructured Meshes for Nonlinear Hyperbolic Systems},
journal = {Journal of Computational Physics},
year = 2007,
volume = 226,
pages = {204--243}, 
}

@Article{CastroToro,
author ={C. C. Castro and E. F. Toro},
title = {Solvers for the High-Order {R}iemann Problem for Hyperbolic Balance Laws},
journal = {Journal of Computational Physics},
year = 2008,
volume = 227,
pages = {2481--2513},
}

@Article{QiuDumbserShu,
  author = 	 {J. Qiu and M. Dumbser and {C.W.} Shu},
  title = 	 {The Discontinuous {Galerkin}  Method with {Lax}-{Wendroff} Type Time Discretizations},
  journal =  {Computer Methods in Applied Mechanics and Engineering},
  year = 	 2005,
  volume =	 194,
  pages =	 {4528--4543}
}

@Article{dumbser_jsc,
  author = 	 {M. Dumbser and {C.D.} Munz},
  title = 	 {Building Blocks for Arbitrary High Order Discontinuous {Galerkin} Schemes},
  journal =      {Journal of Scientific Computing},
  year = 	 2006,
  volume =       27,
  pages =        {215--230}
}

@article{Gassner2011a,
author = {G. Gassner and M. Dumbser and F.  Hindenlang and C.D. Munz},
title = {{Explicit one--step time discretizations for discontinuous {G}alerkin and finite volume schemes based on local predictors}},
journal = {Journal of Computational Physics},
number = {11},
pages = {4232--4247},
volume = {230},
year = {2011}
}

@article{busto2021high,
  title={On high order {ADER} discontinuous {Galerkin} schemes for first order hyperbolic reformulations of nonlinear dispersive systems},
  author={Busto, S. and Dumbser, M. and Escalante, C. and Favrie, N. and Gavrilyuk, S.},
  journal={Journal of Scientific Computing},
  volume={87},
  number={2},
  pages={48},
  year={2021},
  publisher={Springer}
}

@article{gaburro2025high,
  title={{High order Well-Balanced Arbitrary-Lagrangian-Eulerian ADER discontinuous Galerkin schemes on general polygonal moving meshes}},
  author={Gaburro, Elena},
  journal={Computers \& Fluids},
  pages={106764},
  year={2025},
  publisher={Elsevier}
}

@article{han2021dec,
  title={{DeC and ADER}: similarities, differences and a unified framework},
  author={Han Veiga, Maria and {\"O}ffner, Philipp and Torlo, Davide},
  journal={Journal of Scientific Computing},
  volume={87},
  number={1},
  pages={2},
  year={2021},
  publisher={Springer}
}

@article{popov2025high,
  title={High order {ADER-DG method with local DG} predictor for solutions of differential-algebraic systems of equations},
  author={Popov, Ivan S.},
  journal={Journal of Scientific Computing},
  volume={102},
  number={2},
  pages={48},
  year={2025},
  publisher={Springer}
}

@article{lakiss2024ader,
  title={{ADER discontinuous Galerkin material point method}},
  author={Lakiss, Alaa and Heuz{\'e}, Thomas and Tannous, Mikhael and Stainier, Laurent},
  journal={International Journal for Numerical Methods in Engineering},
  volume={125},
  number={1},
  pages={e7365},
  year={2024},
  publisher={Wiley Online Library}
}

@article{dumbser2024well,
  title={{A well-balanced discontinuous Galerkin method for the first--order Z4 formulation of the Einstein--Euler system}},
  author={Dumbser, Michael and Zanotti, Olindo and Gaburro, Elena and Peshkov, Ilya},
  journal={Journal of Computational Physics},
  volume={504},
  pages={112875},
  year={2024},
  publisher={Elsevier}
}

@inproceedings{zanotti2025new,
  title={A new first-order formulation of the {Einstein} equations: comparison among different high order numerical schemes},
  author={Zanotti, Olindo and Dumbser, Michael and Balsara, Dinshaw and Bhoriya, Deepak},
  booktitle={Journal of Physics: Conference Series},
  volume={2997},
  number={1},
  pages={012015},
  year={2025},
  organization={IOP Publishing}
}

@article{rio2024high,
  title={{High-order ADER Discontinuous Galerkin schemes for a symmetric hyperbolic model of compressible barotropic two-fluid flows}},
  author={R{\'\i}o-Mart{\'\i}n, Laura and Dumbser, Michael},
  journal={Communications on Applied Mathematics and Computation},
  volume={6},
  number={4},
  pages={2119--2154},
  year={2024},
  publisher={Springer}
}

@article{ciallella2024very,
  title={{Very high order treatment of embedded curved boundaries in compressible flows: ADER discontinuous Galerkin with a space-time Reconstruction for Off-site data}},
  author={Ciallella, Mirco and Clain, Stephane and Gaburro, Elena and Ricchiuto, Mario},
  journal={Computers \& Mathematics with Applications},
  volume={175},
  pages={1--18},
  year={2024},
  publisher={Elsevier}
}

@article{marot2025mixed,
  title={{Mixed-Precision in High-Order Methods: the Impact of Floating-Point Precision on the ADER-DG Algorithm}},
  author={Marot-Lassauzaie, Marc and Bader, Michael},
  journal={arXiv preprint arXiv:2504.06889},
  year={2025}
}

@article{busto2022new,
  title={A new family of thermodynamically compatible discontinuous {Galerkin} methods for continuum mechanics and turbulent shallow water flows},
  author={Busto, Saray and Dumbser, Michael},
  journal={Journal of Scientific Computing},
  volume={93},
  number={2},
  pages={56},
  year={2022},
  publisher={Springer}
}

@inproceedings{wolf2020optimization,
  title={Optimization and local time stepping of an ader-dg scheme for fully anisotropic wave propagation in complex geometries},
  author={Wolf, Sebastian and Gabriel, Alice-Agnes and Bader, Michael},
  booktitle={International Conference on Computational Science},
  pages={32--45},
  year={2020},
  organization={Springer}
}

@article{dumbser2016high,
  title={High order {ADER} schemes for a unified first order hyperbolic formulation of continuum mechanics: viscous heat-conducting fluids and elastic solids},
  author={Dumbser, Michael and Peshkov, Ilya and Romenski, Evgeniy and Zanotti, Olindo},
  journal={Journal of Computational Physics},
  volume={314},
  pages={824--862},
  year={2016},
  publisher={Elsevier}
}

@article{dumbser2017high,
  title={High order {ADER} schemes for a unified first order hyperbolic formulation of {Newtonian }continuum mechanics coupled with electro-dynamics},
  author={Dumbser, Michael and Peshkov, Ilya and Romenski, Evgeniy and Zanotti, Olindo},
  journal={Journal of Computational Physics},
  volume={348},
  pages={298--342},
  year={2017},
  publisher={Elsevier}
}

@inproceedings{yuan2022hybrid,
  title={Hybrid high-order finite volume discontinuous {Galerkin} methods for turbulent flows},
  author={Yuan, Dean and Tsoutsanis, Panagiotis and Jenkins, K},
  booktitle={World Congress in Computational Mechanics and ECCOMAS Congress},
  year={2022}
}

@article{lorcher2007discontinuous,
  title={{A discontinuous Galerkin scheme based on a space--time expansion. I. Inviscid compressible flow in one space dimension}},
  author={L{\"o}rcher, Frieder and Gassner, Gregor and Munz, C-D},
  journal={Journal of Scientific Computing},
  volume={32},
  number={2},
  pages={175--199},
  year={2007},
  publisher={Springer}
}

@article{popov2025theoryinternalstructureaderdg,
author = {Popov, Ivan S.},
title = {Theory and Internal Structure of {ADER-DG} Method for Ordinary Differential Equations},
journal = {SIAM Journal on Numerical Analysis},
volume = {64},
number = {4},
pages = {1561-1581},
year = {2026},
doi = {10.1137/25M1787008},
URL = { https://doi.org/10.1137/25M1787008},
}

@article{balsara2012self,
  title={Self-adjusting, positivity preserving high order schemes for hydrodynamics and magnetohydrodynamics},
  author={Balsara, Dinshaw S},
  journal={Journal of Computational Physics},
  volume={231},
  number={22},
  pages={7504--7517},
  year={2012},
  publisher={Elsevier}
}

@article{fernandez2022arbitrary,
  title={An arbitrary high order well-balanced {ADER-DG} numerical scheme for the multilayer shallow-water model with variable density},
  author={Fern{\'a}ndez, E Guerrero and D{\'\i}az, MJ Castro and Dumbser, Michael and De Luna, T Morales},
  journal={Journal of Scientific Computing},
  volume={90},
  number={1},
  pages={52},
  year={2022},
  publisher={Springer}
}

@article{rannabauer2018ader,
  title={{ADER-DG} with a-posteriori finite-volume limiting to simulate tsunamis in a parallel adaptive mesh refinement framework},
  author={Rannabauer, Leonhard and Dumbser, Michael and Bader, Michael},
  journal={Computers \& Fluids},
  volume={173},
  pages={299--306},
  year={2018},
  publisher={Elsevier}
}

@article{popov2024effective,
  title={The effective use of {BLAS} interface for implementation of finite-element {ADER-DG} and finite-volume {ADER-WENO} methods},
  author={Popov, Ivan S.},
  number={5},
  journal={Communications in Computational Physics},
  volume={38},
  year={2025},
  pages={1237–1330},
  DOI={10.4208/cicp.OA-2024-0202}
}

@article{veiga2024improving,
  title={On improving the efficiency of {ADER} methods},
  author={Han Veiga, Maria and Micalizzi, Lorenzo and Torlo, Davide},
  journal={Applied Mathematics and Computation},
  volume={466},
  pages={128426},
  year={2024},
  publisher={Elsevier}
}

@article{micalizzi2025efficient,
  title={Efficient iterative arbitrary high-order methods: an adaptive bridge between low and high order},
  author={Micalizzi, Lorenzo and Torlo, Davide and Boscheri, Walter},
  journal={Communications on Applied Mathematics and Computation},
  volume={7},
  number={1},
  pages={40--77},
  year={2025},
  publisher={Springer}
}

@article{offner2025analysis,
	title={Analysis for implicit and implicit-explicit {ADER} and {DeC }methods for ordinary differential equations, advection-diffusion and advection-dispersion equations},
	author={{\"O}ffner, Philipp and Petri, Louis and Torlo, Davide},
	journal={Applied Numerical Mathematics},
	volume={212},
	pages={110--134},
	year={2025},
	publisher={Elsevier}
}

@article{boscheri2022construction,
	title={On the construction of conservative semi-Lagrangian {IMEX} advection schemes for multiscale time dependent {PDEs}},
	author={Boscheri, Walter and Tavelli, Maurizio and Pareschi, Lorenzo},
	journal={Journal of Scientific Computing},
	volume={90},
	number={3},
	pages={97},
	year={2022},
	publisher={Springer}
}

@article{michel2021spectral,
	title={{Spectral analysis of continuous FEM for hyperbolic PDEs: influence of approximation, stabilization, and time-stepping}},
	author={Michel, Sixtine and Torlo, Davide and Ricchiuto, Mario and Abgrall, R{\'e}mi},
	journal={Journal of Scientific Computing},
	volume={89},
	number={2},
	pages={31},
	year={2021},
	publisher={Springer}
}

@article{michel2023spectral,
	title={{Spectral analysis of high order continuous FEM for hyperbolic PDEs on triangular meshes: influence of approximation, stabilization, and time-stepping}},
	author={Michel, Sixtine and Torlo, Davide and Ricchiuto, Mario and Abgrall, R{\'e}mi},
	journal={Journal of Scientific Computing},
	volume={94},
	number={3},
	pages={49},
	year={2023},
	publisher={Springer}
}

@article{friedrich2019entropy_sbp,
	title={Entropy stable space--time discontinuous {Galerkin} schemes with summation-by-parts property for hyperbolic conservation laws},
	author={Friedrich, Lucas and Schn{\"u}cke, Gero and Winters, Andrew R and Fern{\'a}ndez, David C Del Rey and Gassner, Gregor J and Carpenter, Mark H},
	journal={Journal of Scientific Computing},
	volume={80},
	number={1},
	pages={175--222},
	year={2019},
	publisher={Springer}
}

@article{offner2018stability,
	title={Stability of correction procedure via reconstruction with summation-by-parts operators for {Burgers'} equation using a polynomial chaos approach},
	author={{\"O}ffner, Philipp and Glaubitz, Jan and Ranocha, Hendrik},
	journal={ESAIM: Mathematical Modelling and Numerical Analysis},
	volume={52},
	number={6},
	pages={2215--2245},
	year={2018},
	publisher={EDP Sciences}
}

@article{gaburro2023high,
	title={High order entropy preserving {ADER-DG }schemes},
	author={Gaburro, Elena and {\"O}ffner, Philipp and Ricchiuto, Mario and Torlo, Davide},
	journal={Applied Mathematics and Computation},
	volume={440},
	pages={127644},
	year={2023},
	publisher={Elsevier}
}

@article{ranocha2020relaxation,
	title={{Relaxation Runge--Kutta methods: Fully discrete explicit entropy-stable schemes for the compressible Euler and Navier--Stokes equations}},
	author={Ranocha, Hendrik and Sayyari, Mohammed and Dalcin, Lisandro and Parsani, Matteo and Ketcheson, David I},
	journal={SIAM Journal on Scientific Computing},
	volume={42},
	number={2},
	pages={A612--A638},
	year={2020},
	publisher={SIAM}
}

@article{abgrall2022relaxation,
	title={Relaxation deferred correction methods and their applications to residual distribution schemes},
	author={Abgrall, R{\'e}mi and Le M{\'e}l{\'e}do, {\'E}lise and {\"O}ffner, Philipp and Torlo, Davide},
	journal={The SMAI Journal of computational mathematics},
	volume={8},
	pages={125--160},
	year={2022}
}

@article{muzzolon2025high,
title = {High order numerical discretizations of the {Einstein-Euler} equations in the generalized harmonic formulation},
journal = {Journal of Computational Physics},
volume = {563},
pages = {115084},
year = {2026},
issn = {0021-9991},
doi = {https://doi.org/10.1016/j.jcp.2026.115084},
url = {https://www.sciencedirect.com/science/article/pii/S0021999126004377},
author = {Stefano Muzzolon and Michael Dumbser and Olindo Zanotti and Elena Gaburro},
}

@article{ricardo2025scalable,
  title={Scalable {ADER-DG} Transport Method with Polynomial Order Independent {CFL} Limit},
  author={Ricardo, Kieran and Duru, Kenneth},
  journal={arXiv preprint arXiv:2507.07304},
  year={2025}
}

@Article{Gaburro2021,
  author    = {Gaburro, Elena and Dumbser, Michael},
  title     = {A posteriori subcell finite volume limiter for general {PNPM} schemes: applications from gasdynamics to relativistic magnetohydrodynamics},
  journal   = {Journal of Scientific Computing},
  year      = {2021},
  volume    = {86},
  number    = {3},
  pages     = {37},
  publisher = {Springer},
}

@article{dumbser2016space,
  title={A space-time discontinuous {Galerkin} method for {Boussinesq}-type equations},
  author={Dumbser, Michael and Facchini, Matteo},
  journal={Applied Mathematics and Computation},
  volume={272},
  pages={336--346},
  year={2016},
  publisher={Elsevier}
}

@article{boscheri2017arbitrary,
  title={{Arbitrary-Lagrangian--Eulerian discontinuous Galerkin} schemes with a posteriori subcell finite volume limiting on moving unstructured meshes},
  author={Boscheri, Walter and Dumbser, Michael},
  journal={Journal of Computational Physics},
  volume={346},
  pages={449--479},
  year={2017},
  publisher={Elsevier}
}

@article{jackson2017eigenvalues,
  author={H. Jackson}, 
  title={On the eigenvalues of the {ADER-WENO} {G}alerkin predictor}, 
  journal={Journal of Computational Physics},
  volume={333},
  pages={409--413},
  year={2017},
}

@article{cockburn1998local,
	title={The local discontinuous {Galerkin} method for time-dependent convection-diffusion systems},
	author={Cockburn, Bernardo and Shu, Chi-Wang},
	journal={SIAM journal on numerical analysis},
	volume={35},
	number={6},
	pages={2440--2463},
	year={1998},
	publisher={SIAM}
}

@article{liu20082,
	title={L2 stability analysis of the central discontinuous {Galerkin} method and a comparison between the central and regular discontinuous {Galerkin} methods},
	author={Liu, Yingjie and Shu, Chi-Wang and Tadmor, Eitan and Zhang, Mengping},
	journal={ESAIM: Mathematical Modelling and Numerical Analysis},
	volume={42},
	number={4},
	pages={593--607},
	year={2008},
	publisher={EDP Sciences}
}

@incollection{sherwin2000dispersion,
	title={Dispersion analysis of the continuous and discontinuous {Galerkin} formulations},
	author={Sherwin, Spencer},
	booktitle={Discontinuous Galerkin Methods: Theory, Computation and Applications},
	pages={425--431},
	year={2000},
	publisher={Springer}
}

@article{may2017explicit,
  title={An explicit implicit scheme for cut cells in embedded boundary meshes},
  author={May, Sandra and Berger, Marsha},
  journal={Journal of Scientific Computing},
  volume={71},
  number={3},
  pages={919--943},
  year={2017},
  publisher={Springer}
}

@article{may2024accuracy,
  title={Accuracy analysis for explicit-implicit finite volume schemes on cut cell meshes},
  author={May, Sandra and Laakmann, Fabian},
  journal={Communications on Applied Mathematics and Computation},
  volume={6},
  number={4},
  pages={2239--2264},
  year={2024},
  publisher={Springer}
}

@article{yan2002local1,
  title={A local discontinuous Galerkin method for KdV type equations},
  author={Yan, Jue and Shu, Chi-Wang},
  journal={SIAM Journal on Numerical Analysis},
  volume={40},
  number={2},
  pages={769--791},
  year={2002},
  publisher={SIAM}
}

@article{yan2002local2,
  title={Local discontinuous Galerkin methods for partial differential equations with higher order derivatives},
  author={Yan, Jue and Shu, Chi-Wang},
  journal={Journal of Scientific Computing},
  volume={17},
  number={1},
  pages={27--47},
  year={2002},
  publisher={Springer}
}

@article{popov2026,
      title={Theory and internal structure of ADER-DG method for partial differential equations}, 
      author={I. S. Popov},
  journal={arXiv preprint arXiv:2606.17100},
  year={2026},
}

@Article{Gaburro2027,
  author        = {Gaburro, Elena and Kl{\'\i}ma, Mat{\v{e}}j and Bonafini, Mauro and Tavelli, Maurizio},
  title         = {On the treatment of topology changes on 3D polyhedral moving meshes via 4D space--time hole-like elements in direct ALE ADER-DG methods},
  journal       = {Computer Methods in Applied Mechanics and Engineering},
  year          = {2027},
  volume        = {463},
  pages         = {119434},
  __markedentry = {[mauro:6]},
  publisher     = {Elsevier},
}

\end{document}